\documentclass{article}
\usepackage{amsmath,amssymb,amsthm,pb-diagram,multirow}
\usepackage[all]{xy}
\usepackage{graphicx}
\usepackage[dvipdfmx]{color}
\usepackage[initials]{amsrefs}

\newtheorem{main}{Theorem}

\newtheorem{main_cor}[main]{Corollary}
\newtheorem{theorem}{Theorem}[section]
\newtheorem{proposition}[theorem]{Proposition}
\newtheorem{lemma}[theorem]{Lemma}
\newtheorem{corollary}[theorem]{Corollary}

\newtheorem{question}{Question}

\theoremstyle{definition}

\def\A{\mathbb{A} }

\def\R{\mathbb{R} }
\def\Z{\mathbb{Z} }
\def\nbd{neighborhood }
\def\nbds{neighborhoods }

\def\Sv{\mathop{\mathrm{Sing}}(v)}
\def\Pv{\mathop{\mathrm{Per}}(v)}
\def\Cv{\mathop{\mathrm{Cl}}(v)}

\begin{document}

\begin{center}
\textbf{Vladislav Kibkalo$^1$, Tomoo Yokoyama$^{2}$}
\end{center}

\begin{center}
\textbf{Topological characterizations of Morse-Smale flows on surfaces and generic non-Morse-Smale flows}
\end{center}

\begin{center}
37D15,37B35,37E35,37C75,37B25
\end{center}

\begin{center}
$^1$ Lomonosov Moscow State University, Faculty of Mechanics and Mathematics, Russia,  \\ 
slava.kibkalo@gmail.com   \\
$^2$ Applied Mathematics and Physics Division, Gifu University, Yanagido 1-1, Gifu, 501-1193, Japan\\
tomoo@gifu-u.ac.jp

\end{center}


\begin{abstract}
It is known that $C^r$ Morse-Smale vector fields form an open dense subset in the space of vector fields on orientable closed surfaces and are structurally stable for any $r \in \mathbb{Z}_{>0}$. In particular, $C^r$ Morse vector fields (i.e. Morse-Smale vector fields without limit cycles) form an open dense subset in the space of $C^r$ gradient vector fields on orientable closed surfaces and are structurally stable. Therefore generic time evaluations of gradient flows on orientable closed surfaces (e.g. solutions of differential equations) are described by alternating sequences of Morse flows and instantaneous non-Morse gradient flows. To illustrate the generic transitions, we characterize and list all generic non-Morse gradient flows. To construct such characterizations, we characterize isolated singular points of gradient flows on surfaces. In fact,  such a singular point is a non-trivial finitely sectored singular point without elliptic sectors.   Moreover, considering Morse-Smale flows as ``generic gradient flows with limit cycles'', we characterize and list all generic non-Morse-Smale flows.
\end{abstract}


\textbf{Keywords:} Morse-Smale flows, transitions, complete invariant.

\section{Introduction}
\subsection{Genericity and stability of Morse-Smale flows on surfaces}
By a work of Andronov-Pontryagin~\cite{andronov1937rough} and a work of Peixoto~\cite{peixoto1962structural}, it is known that the set $\Sigma^r(S)$ of Morse-Smale $C^r$-vector fields ($r \geq 1$) on an orientable closed surface $S$ is open dense in the set $\chi^r(S)$ of $C^r$-vector fields on $S$ and that Morse-Smale $C^r$-vector fields are structurally stable in  $\chi^r(S)$.
Combining Pugh's $C^1$-Closing Lemma~\cites{pugh1967improved,pugh1983c}, the subset $\Sigma^1$ for a non-orientable closed surface $M$ is dense in $\chi^1(M)$.
When the non-orientable genus of a non-orientable closed surface $M$ is less than $5$, the subset $\Sigma^r(M)$ is open dense in $\chi^r(M)$, and Morse-Smale $C^r$-vector fields are structurally stable in $\chi^r(M)$~\cites{gutierrez1978structural,markley1970number}.
Moreover, Smale \cite{smale1961gradient} proved that any Morse flow (i.e. Morse-Smale flows without limit cycles) on a closed manifold is a gradient flows without separatrices from a saddle to a saddle.
By this fact, since the set of Morse vector fields is open dense in the set of gradient vector fields, we characterize a ``generic'' non-Morse gradient flow on a compact surface to describe a generic time evaluation of gradient flows on orientable compact surfaces (e.g. solutions of differential equations) which is an alternating sequence of Morse flows and instantaneous non-Morse gradient flows.

\subsection{Characterization of isolated singular points in gradient flows on surfaces}
Cobo, Gutierrez, and Llibre proved that the singular points of a non-wandering flow with finitely many singular points on a compact surface are either centers or multi-saddles (see \cite{cobo2010flows}*{Theorem~3}).
Conversely, a quasi-regular flow on a compact surface has finitely many singular points.
On the other hand, we consider the following analogous question for hyperbolic dynamics.
\begin{question}
What kinds of isolated singular points do appear in gradient flows on manifolds?
\end{question}

In the surface case, we characterize isolated singular points in gradient flows on surfaces as follows.

\begin{main}\label{main:a}
An isolated singular point of a gradient flow on a surface is a non-trivial finitely sectored singular point without elliptic sectors.
\end{main}

The previous theorem is a key observation to characterize generic non-Morse gradient flows on compact surfaces.


\subsection{Topological characterizations of gradient flows, Morse flows, and Morse flows on surfaces}
To characterize generic non-Morse gradient flows on surfaces, we topologically characterize gradient flows with finitely many singular points on compact surfaces and recall some concepts.

The restriction $v|_A$ for a subset $A$ is a sector for a singular point $x$ if there are a non-degenerate interval $I \subseteq [0, 2 \pi )$ and a homeomorphism $h \colon \{ x\} \sqcup A \to \{ 0 \} \sqcup \{ ( r \cos \theta, r \sin \theta ) \in \R^2 \mid r \in (0,1), \theta \in I \}$ such that $h^{-1}(0) = x$.
A parabolic sector is topologically equivalent to a flow box with the point $(\pm \infty, 0)$ and a hyperbolic (resp. elliptic) sector is topologically equivalent to a Reeb component (resp. the interior of a Reeb component) with the point $(\infty, 0)$ (resp. $(- \infty, 0)$) as in Figure~\ref{sectors_all}.
A separatrix is a non-singular orbit from or to a singular point.
A separatrix is a hyperbolic (resp. elliptic, parabolic) border separatrix if it is contained in the boundary of a hyperbolic sector (resp. a maximal open elliptic sector, a maximal open parabolic sector).
A singular point is finitely sectored if either it is a center or there is its open neighborhood which is an open disk and is a finite union of the point, parabolic sectors, hyperbolic sectors, and elliptic sectors such that a pair of distinct sectors intersects at most two orbit arcs.
A flow $v$ on a surface $S$ is of weakly finite type if each recurrent point is closed, there are at most finitely many limit cycles, and each singular point is finitely sectored.
A flow $v$ of weakly finite type is Morse-Smale-like if the set of non-recurrent points is open dense in $S$.
%
Then we have the following characterization of gradient flow with finitely many singular points on a compact surface.
\begin{figure}
\begin{center}
\includegraphics[scale=0.2]{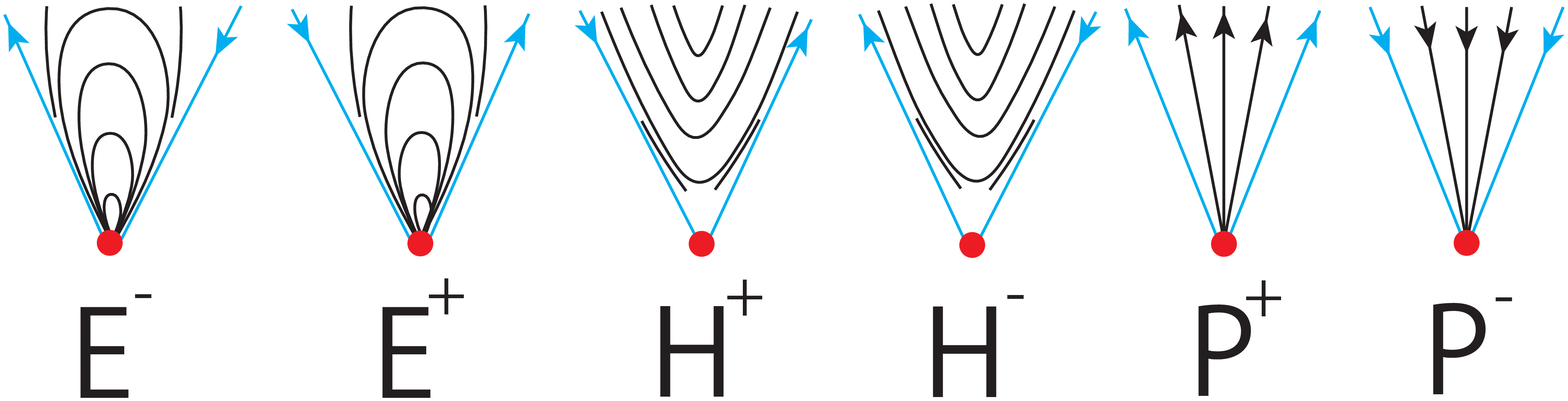}
\end{center}
\caption{Two parabolic sectors $P^-$ and $P^+$, two hyperbolic sectors $H^-$ and $H^+$ with clockwise and anti-clockwise orbit directions, and two elliptic sectors $E^+$ and $E^-$ with clockwise and anti-clockwise orbit directions respectively.}
\label{sectors_all}
\end{figure}

\begin{main}\label{main:c}
The following are equivalent for a flow on a compact surface:
\\
$(1)$ The flow is a gradient flow with finitely many singular points.
\\
$(2)$ The flow is Morse-Smale-like without elliptic sectors or non-trivial circuits.
\end{main}

A flow is quasi-regular if it is topological equivalent to a flow with multi-saddles as in Figure~\ref{multi-saddles}, centers, ($\partial$-)sink, and ($\partial$-)sources as in Figure~\ref{non-deg-sing}.
Notice that quasi-regularity implies that each singular point is a finitely sectored singular point without elliptic sectors.
The quasi-regularity implies the following statement.
\begin{figure}
\begin{center}
\includegraphics[scale=0.55]{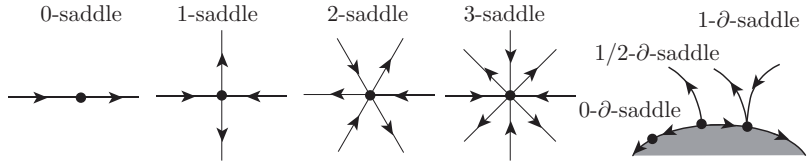}
\end{center}
\caption{Multi-saddles.}
\label{multi-saddles}
\end{figure}

\begin{figure}
\begin{center}
\includegraphics[scale=0.35]{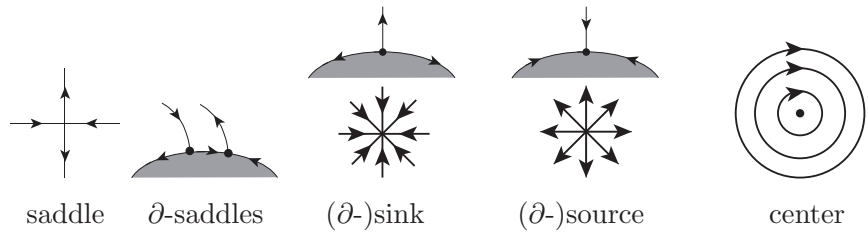}
\end{center}
\caption{A saddle, two $\partial$-saddles, a sink, a $\partial$-sink, a source, a $\partial$-source, and a center}
\label{non-deg-sing}
\end{figure}

\begin{main_cor}\label{main_cor}
The following are equivalent for a quasi-regular flow on a compact surface:
\\
$(1)$ The flow is gradient.
\\
$(2)$ The flow is Morse-Smale-like without non-trivial circuits.
\\
$(3)$ The flow is a flow without non-trivial circuits such that any recurrent points are closed and that the set of non-recurrent points are open dense.
\end{main_cor}

The non-existence of non-trivial circuits is necessary. In fact, there is a non-gradient Morse-Smale-like flow without limit circuits on a Klein bottle (see Figure~\ref{nonori_ex}).
\begin{figure}
\begin{center}
\includegraphics[scale=0.3]{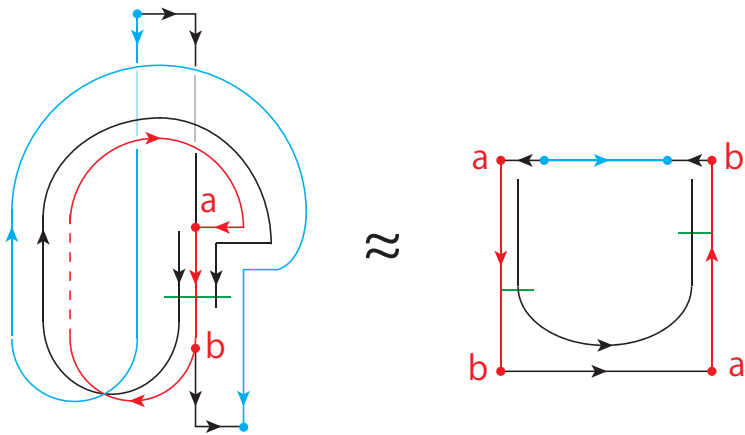}
\end{center}
\caption{The resulting flow on the double of the M\"obius band is a non-gradient Morse-like flow without limit circuits on a Klein bottle}
\label{nonori_ex}
\end{figure}

A flow is regular (or has non-degenerate singular points) if it is topological equivalent to a flow with non-degenerate singular points (i.e ($\partial$-)saddles, ($\partial$-)sink, ($\partial$-)sources), and centers as in Figure~\ref{non-deg-sing}.
A separatrix is heteroclinic if it connects distinct singular points.
A Morse-Smale-like flow $v$ on a surface $S$ is Morse-like if there are no limit cycles.
We have the following characterization of Morse flows on compact surfaces.

\begin{main}\label{main:Morse}
The following are equivalent for a flow $v$ on a compact surface $S$:
\\
$(1)$ The flow $v$ is Morse.
\\
$(2)$ The flow $v$ is a regular Morse-like flow such that each multi-saddle separatrix is contained in the boundary $\partial S$.
\\
$(3)$ The flow $v$ is a regular Morse-Smale-like flow without nontrivial circuits such that each heteroclinic multi-saddle separatrix is contained in the boundary $\partial S$.
\end{main}

A separatrix is a multi-saddle separatrix if it connects multi-saddles.
Considering Morse-Smale flows to be ``generic gradient flows with limit cycles'', called quasi-Morse-Smale flow, we characterize Morse-Smale flows as follows.

\begin{main}\label{main:d}
The following are equivalent for a flow on a compact surface $S$:
\\
$(1)$ The flow is Morse-Smale.
\\
$(2)$ The flow $v$ is a regular Morse-Smale-like flow such that each limit cycle is topologically hyperbolic and that each multi-saddle separatrix is contained in the boundary $\partial S$.
\\
$(3)$ The flow is a regular quasi-Morse-Smale flow such that each limit cycle is topologically hyperbolic and that each multi-saddle separatrix is contained in the boundary $\partial S$.
\\
$(4)$ The closed point set consists of finitely many topologically hyperbolic orbits, any recurrent points are closed, and each multi-saddle separatrix is contained in the boundary $\partial S$.
\end{main}

\subsection{Characterization of ``gradient flows with limit cycles'' on compact surfaces}

As mentioned above, a Morse flow is a gradient flows without multi-saddle separatrices \cite{smale1961gradient}.
Similarly, the resulting flow from a Morse-Smale flow by cutting limit cycles and by collapsing new boundary components into singletons is a gradient (cf. Theorem~\ref{main:d}).
Moreover, we show that so is the resulting flow from a quasi-regular Morse-Smale flow by cutting non-trivial circuits and by collapsing new boundary components into singletons (see Theorem~\ref{main:qausi-MS}).
More precisely, we recall some concepts to state the statement.

By a non-trivial circuit, we mean either a cycle or a continuous image of a directed cycle which is a graph but not a singleton, whose orientations of edges correspond to the directions of orbits, and which is the union of separatrices and finitely many singular points.
In other words, a non-trivial circuit which is not a cycle, is a directed path as a graph whose initial point is also terminal.
Roughly speaking, a non-trivial holonomy along a non-trivial circuit $\gamma$ corresponds with the first return map on a transverse near $\gamma$ and non-triviality of the holonomy corresponds with non-identity of the first return map.
A non-trivial circuit $\gamma$ has non-trivial holonomy if there are open transverse arcs $I$ and $J$ whose boundary component is a point $x \in \gamma$ and there is an open flow box $U$ whose vertical boundary components $\partial_\perp^- U$ and $\partial_\perp^+ U$ are $I$ and $J$ (i.e. $\partial_\perp^- U = I$ and $\partial_\perp^+ U = J$) such that $\gamma$ is a connected component of the difference $\partial U - (\partial_\perp^- U \cup \partial_\perp^+ U)$ and that the first return map $f_v: I \to J$ is not identical (see Figure~\ref{hol}).
\begin{figure}
\begin{center}
\includegraphics[scale=0.35]{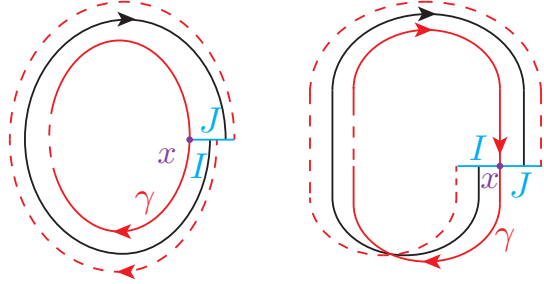}
\end{center}
\caption{The figure on the left is an orientable holonomy from $I$ to $J$ along a non-trivial circuit $\gamma$, and the figure on the right is a non-orientable holonomy from $I$ to $J$ along a non-trivial circuit $\gamma$}
\label{hol}
\end{figure}%

Let $v$ be a flow on a surface $S$ and $\Gamma \subseteq \Sv \sqcup \mathrm{P}(v)$ the union of non-trivial circuits with non-trivial holonomy.
Denote by $S_{\mathrm{me}}$ the metric completion of the difference $S - \Gamma$.
Let $p_{\mathrm{me}}: S_{\mathrm{me}} \to S$ be the canonical projection and the lift $\Gamma_{\mathrm{me}} := p_{\mathrm{me}}^{-1}(\Gamma)$.
Note that if $\Gamma$ consists of finitely many orbits then $S - \Gamma = S_{\mathrm{me}} - \Gamma_{\mathrm{me}}$ and the restriction $p_{\mathrm{me}}|_{\Gamma_{\mathrm{me}}}: \Gamma_{\mathrm{me}} \to \Gamma$ is an immersion.
By constriction, the restriction $p_{\mathrm{me}}|_{\Gamma_{\mathrm{me}} \setminus p_{\mathrm{me}}^{-1}(\Sv)}: \Gamma_{\mathrm{me}} \setminus p_{\mathrm{me}}^{-1}(\Sv) \to \Gamma \setminus \Sv$ is an immersion.
The induced flow $v_{\mathrm{me}}$ on $S_{\mathrm{me}}$ is well-defined such that $v_{\mathrm{me}} = v$ on the complement $S - \Gamma = S_{\mathrm{me}} - \Gamma_{\mathrm{me}}$ and that the image of any orbit in $\Gamma_{\mathrm{me}}$ with respect to $v_{\mathrm{me}}$ is an orbit with respect to $v$.
Let $S_{\mathrm{col}}$ be the resulting surface from $v_{\mathrm{me}}$ by collapsing limit circuits into singletons,  $v_{\mathrm{col}}$ the resulting flow on $S_{\mathrm{col}}$, and $p_{\mathrm{col}}: S_{\mathrm{me}} \to S_{\mathrm{col}}$ the canonical projection (see Figure~\ref{Fig:blowdown}).
\begin{figure}
\[
\xymatrix@=18pt{
S \ar@{}[d]|{\bigcup} & & S_{\mathrm{me}}\ar[ll]_{p_{\mathrm{me}}} \ar[rr]^{p_{\mathrm{col}}} \ar@{}[d]|{\bigcup} & &  S_{\mathrm{col}} \ar@{}[d]|{\bigcup} \\
S - \Gamma & &  S_{\mathrm{me}} - \Gamma_{\mathrm{me}} \ar@{=}[ll]_{p_{\mathrm{me}}|}
 \ar@{=}[rr]^{p_{\mathrm{col}}|} & &  S_{\mathrm{col}} - \Gamma_{\mathrm{col}}
 }
\]
\caption{Canonical quotient mappings induced by the metric completion and the collapse for the case that $\Gamma$ consists of finitely many orbits.}
\label{Fig:blowdown}
\end{figure}
Put $\Gamma_{\mathrm{col}} := p_{\mathrm{col}}(\Gamma_{\mathrm{me}}) = p_{\mathrm{col}}(p_{\mathrm{me}}^{-1}(\Gamma))$.
A flow $v$ on a surface $S$ is quasi-Morse-Smale if the resulting flow $v_{\mathrm{col}}$ is a gradient flow.
Roughly speaking, quasi-regular Morse-Smale-like flows can be considered as ``gradient flows with non-trivial circuits''.
Precisely, we have the following equivalence of quasi-Morse-Smale flows and Morse-Smale-like flows under quasi-regularity and finite existence of limit cycles.

\begin{main}\label{main:qausi-MS}
The following are equivalent for a quasi-regular flow $v$ on a compact surface:
\\
{\rm(1)} The flow $v$ is Morse-Smale-like.
\\
{\rm(2)} The flow $v$ is quasi-Morse-Smale and has at most finitely many limit cycles.
\end{main}

\subsection{Generic non-gradient and non-Morse-Smale flows on surfaces}

Notice that the set of $C^1$ gradient flows is not open in the set of $C^1$ flows because singular points need not non-degenerate.
Moreover, limit cycles for $C^1$ flows can be bifurcated into topologically non-hyperbolic limit cycles.
However, forbidding the existence of fake limit cycles, fixing the sum of indices of sources, sinks, $\partial$-sources, and $\partial$-sinks, and fixing the number of limit cycles, we topologically characterize generic intermediate flows of gradient flows and of Morse-Smale flows.
Here a fake limit cycle is a limit cycle that is semi-attracting on a side and semi-repelling on another side.
In fact, roughly speaking, generic intermediate flows between gradient flows on compact surfaces can be obtained from Morse flows either by merging exactly one pair of separatrices from/to multi-saddles or by pinching exactly two $\partial$-saddles unless annihilations and creations of singular points occur.
Similarly, generic intermediate flows between Morse-Smale flows on compact surfaces can be obtained from Morse-Smale flows either by merging exactly one pair of separatrices from/to multi-saddles or by pinching exactly two $\partial$-saddles unless annihilations and creations of singular points and limit cycles occur.
More precisely, we recall some concepts to state the statement.

Denote by $\mathrm{P}_{\mathrm{ms}}(v)$ the union of multi-saddle separatrices.
A singular point is attracting if it is either a sink or a $\partial$-sink, and is repelling if it is either  a source, and a $\partial$-source.
A singular point is a fake saddle if it is either $0$-saddle or $0$-$\partial$-saddle.
A parabolic sector $A$ for a singular point $x$ is fake if there are hyperbolic border separatrices $\gamma, \mu$ from/to $x$ such that $A$ is transversely bounded by $\gamma, \mu$ as in Figure~\ref{fake_parabolic_sector}.
\begin{figure}
\begin{center}
\includegraphics[scale=0.15]{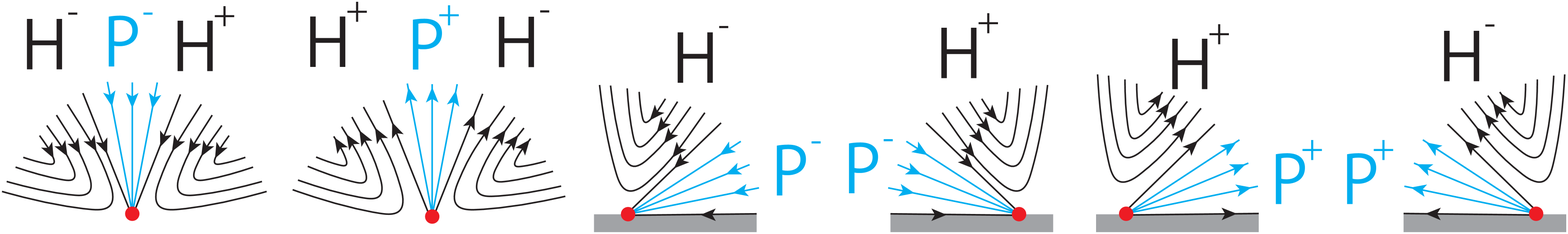}
\end{center}
\caption{Fake parabolic separatrices}
\label{fake_parabolic_sector}
\end{figure}%
Roughly speaking, a fake parabolic sector is a parabolic sector either between two hyperbolic sectors or between one hyperbolic sector and a boundary component.
Theorem~\ref{main:a} implies that a gradient flow with isolated singular point on a surface is quasi-regular if and only if there are no fake parabolic sectors.

For any $k \in \mathbb{Z}_{>0}$ and $r \in \mathbb{Z}_{\geq 0}$, denote by $\mathcal{G}^r_{k/2,*}(S)$ the set of quasi-regular gradient $C^r$ flows without fake saddles on a compact surface $S$ whose sum of indices of attracting/repelling singular points (i.e. sinks, $\partial$-sinks, sources, and $\partial$-sources) is $k/2$.
Equip $\mathcal{G}^r_{k/2,*}(S)$ with the $C^r$ topology.
%
Denote by $\mathcal{M}^r_{k/2}(S)$ the set of Morse $C^r$ flows whose sum of indices of attracting/repelling singular points is $k/2$.
We state the open density of Morse flows under conditions that there are no annihilations/creations of singular points and of limit cycles.

\begin{main}\label{th:open_dense_M}
For any $r \in \mathbb{Z}_{> 0}$,  the set of Morse flows on a compact surface $S$ in $\mathcal{G}^r_{k/2,*}(S)$ is open dense in $\mathcal{G}^r_{k/2,*}(S)$.
\end{main}

A regular flow in $\mathcal{G}_{k/2,*}(S)$ is $h$-unstable if there is exactly one multi-saddle separatrix outside of the boundary $\partial S$.
Roughly speaking, it is known that $\Omega$-stable flows on closed surfaces are ``Morse-Smale'' flow without the non-existence condition of heteroclinic separatrices.
In fact,
a $C^1$ vector field on a closed manifold is $\Omega$-stable if and only if it satisfies the no-cycle condition  and all non-wandering orbits are hyperbolic closed orbits \cite{hayashi1997connecting,pugh1970omega}.
Therefore any $\Omega$-stable flows on closed surfaces correspond to ``Morse-Smale'' flows without the non-existence condition of heteroclinic separatrices, and so any $h$-unstable gradient flows are $\Omega$-stable.
A flow $v \in \mathcal{G}_{k/2,*}(S)$ with $\mathrm{P}_{\mathrm{ms}}(v) \subset \partial S$ is $p$-unstable if all singular points except one $1$-$\partial$-saddle are topologically hyperbolic.
By definition, notice that the $p$-unstable flows are not $\Omega$-stable.
Put $\mathcal{M}_{k/2}^r(S)^c := \mathcal{G}_{k/2,*}^r(S) - \mathcal{M}_{k/2}^r(S)$.
Roughly speaking, the following theorem says that generic intermediate flows between gradient flows on compact surfaces can be obtained from Morse flows either by merging exactly one pair of separatrices from/to multi-saddles or by pinching exactly two $\partial$-saddles unless annihilations and creations of singular points occur.
In other words, we describe generic transitions between gradient flows with Morse flows on a compact surface  under the non-existence of annihilations or creations of singular points, the quasi-regularity for singular points,  and the non-existence of fake saddles and fake parabolic sectors.

\begin{main}\label{main:e}
For any $r \in \mathbb{Z}_{\geq 0}$, the $h$-unstable flows and $p$-unstable flows in $\mathcal{M}_{k/2}^r(S)^c$ forms an open dense subset in $\mathcal{M}_{k/2}^r(S)^c$.
\end{main}

We have the following complete list of generic transitions between Morse flows.

\begin{main}\label{th:generic_gradient_trans}
For any $r \in \mathbb{Z}_{>0}$, The transitions among gradient flows in $\mathcal{G}^r_{k/2,*}(S)$  listed as in Figure~\ref{trans_all} forms an open dense subset in $\mathcal{M}_{k/2}^r(S)^c$.
\begin{figure}
\begin{center}
\includegraphics[scale=0.3]{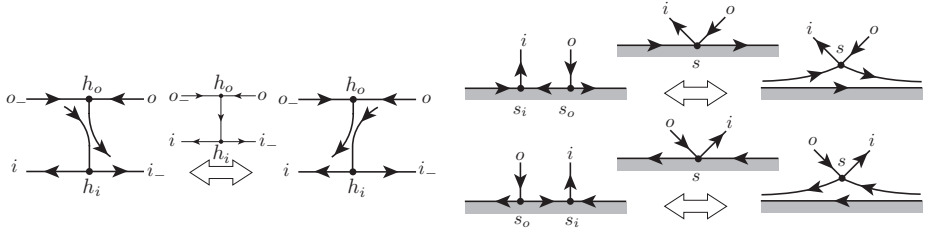}
\end{center}
\caption{Left, all heteroclinic separatrices between a saddle/$\partial$-saddle $h_o$ and a saddle/$\partial$-saddle $h_i$ which are not contained in limit circuits of $h$-unstable flows; right, all pinching structures in $p$-unstable flows.}
\label{trans_all}
\end{figure}
In particular, there are exactly four kinds of heteroclinic separatrices in $h$-unstable on the left in Figure~\ref{trans_all} and exactly two kinds of pinching structures in $p$-unstable flows on the right in Figure~\ref{trans_all}.
\end{main}

Notice that generic local transitions on non-orientable surfaces can occur (e.g. Figure~\ref{nonori_hol}).

For any $k \in \mathbb{Z}_{>0}$, $r \in \mathbb{Z}_{\geq 0}$, and $l \in \mathbb{Z}_{>0}$, denote by $\mathcal{Q}_{k/2,l,**}^r(S)$ the set of quasi-regular Morse-Smale-like flows without fake saddles or fake limit cycles whose sum of indices of attracting/repelling singular points is $k/2$ and whose numbers of limit circuits are $l$.
From Theorem~\ref{main:qausi-MS}, any flow in $\mathcal{Q}_{k/2,l,**}^r(S)$ is quasi-Morse-Smale.
Denote by $\mathcal{MS}_{k/2,l}^r(S)$ the set of $C^r$ Morse-Smale flows whose sum of indices of attracting/repelling singular points is $k/2$ and each of whose number of limit circuits is $l$.
We state the open density of Morse-Smale flows under the non-existence of annihilations or creations of singular points and limit cycles, the quasi-regularity for singular points, and the non-existence of fake saddles and fake parabolic sectors

\begin{main}\label{th:open_dense_MS}
For any $r \in \mathbb{Z}_{>0}$, the set of Morse-Smale flows on a compact surface $S$ in $\mathcal{Q}^r_{k/2,l,**}(S)$ is open dense in $\mathcal{Q}^r_{k/2,l,**}(S)$.
\end{main}

Put $\mathcal{MS}_{k/2,l}^r(S)^c := \mathcal{Q}_{k/2,l,**}^r(S) - \mathcal{MS}_{k/2,l}^r(S)$.
We generalize $h$-unstability and $p$-unstabiunstabilitylity to Morse-Smale-like flows.
A regular flow in $\mathcal{MS}_{k/2,l}^r(S)^c$ is $h$-unstable if there is exactly one multi-saddle separatrix outside of the boundary $\partial S$.
As above, any $h$-unstable Morse-Smale-like flows are $\Omega$-stable.
A flow $v \in \mathcal{MS}_{k/2,l}^r(S)^c$ with $\mathrm{P}_{\mathrm{ms}}(v) \subset \partial S$ is $p$-unstable if all singular points except one $1$-$\partial$-saddle are topologically hyperbolic.
As the gradient flow case, generic time evolutions of $C^1$ ``gradient flows with limit cycles'' can be described in $\mathcal{Q}^1_{k/2,l}(S)$.

\begin{main}\label{main:f}
For any $r \in \mathbb{Z}_{>0}$, the $h$-unstable flows and $p$-unstable flows in $\mathcal{MS}_{k/2,l}^r(S)^c$ forms an open dense subset in $\mathcal{MS}_{k/2,l}^r(S)^c$.
\end{main}

Moreover, we list all ``generic'' intermediate flows between Morse-Smale flows.

\begin{main}\label{th:generic_MS_trans}
For any $r \in \mathbb{Z}_{>0}$, The transitions between Morse-Smale flows in $\mathcal{Q}^r_{k/2,l,**}(S)$ listed as in Figure~\ref{trans_all}--\ref{trans_limit_all} forms an open dense subset in $\mathcal{MS}_{k/2,l}^r(S)^c$.
\begin{figure}
\begin{center}
\includegraphics[scale=0.2]{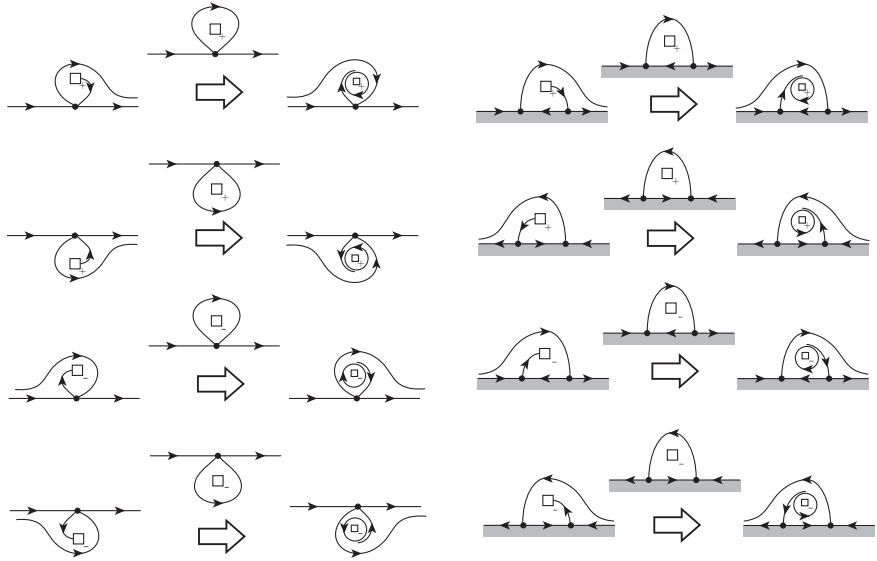}
\end{center}
\caption{Left, all homoclinic separatrices of saddles in $h$-unstable flows; right, all heteroclinic separatrices of $\partial$-saddles in limit circuits of $h$-unstable flows.}
\label{trans_limit_all}
\end{figure}
\end{main}

%
%

\subsection{Topological invariants of flows on surfaces}
There are several topological invariants of flows on surfaces.
For instance, a fundamental result of Morse theory says that gradient flows of Morse functions on closed surfaces are characterized by the set of saddles and their separatrices, which are finite directed graphs.
The Morse theory for gradient vector fields on compact manifolds is extended to an index theory for Smale flows on compact manifolds using Lyapunov graphs \cite{franks1985nonsingular}, which are generalizations of the quotient spaces of gradient functions.
Recall that a flow is Smale if (1) its chain recurrent set has a hyperbolic structure and its dimension is less than or equal to one; and (2) it satisfies the transverse condition.
In \cite{de1993lyapunov}, a characterization of Lyapunov graphs associated to smooth flows on surfaces is presented.
Moreover, Peixoto graph is a complete invariant for Morse flows.
%

A three-colour graph  graph is a complete invariant for Morse-Smale flows \cite{oshemkov1998classification}. %
The latter construction was also inspired by invariants of Hamiltonian systems with non-degenerate critical points (saddles or centers) on a two-dimensional closed manifold and of multi-dimensional integrable Hamiltonian systems. The theory of their classifying invariants was developed by Fomenko and his school, see \cites{Fomenko87, Bolsinov90, Fomenko91} and, in great detail, see \cite{bolsinov2004integrable}.

Oshemkov's construction of a three-coloured graph was generalized to the case of saddle-saddle connections, i.e. for flows that are not of Morse or Morse-Smale type \cite{kruglov2018topological}. An equipped multi-coloured graph is a complete invariant for $\Omega$-stable flows on closed surfaces, each of which is a ``Morse-Smale'' flow without the non-existence condition of heteroclinic separatrices.

Some other invariants for non-compact Hamiltonian case were constructed by Nikolaenko \cite{Nikolaenko20}. They generalized construction of $f$-graph suggested by Oshemkov \cite{Oshemkov94} for Hamiltonian systems with a finite number of Morse singular points to encode efficiently its semi-local singularities, i.e. systems in a small neighbourhood of a connected component of energy level that contains saddles. Such $f$-graphs were later generalized by himself \cite{Oshemkov10} to classify non-degenerate singularities of Hamiltonian systems with more that 1 degree of freedom.

As it turns out, another representation of $f$-graph in terms of some permutations has a very natural realization. Such invariants appear in a class of integrable billiards in piece-wise flat domains that are CW-complexes with permutations introduced and studied in recent years by Vedyushkina and Fomenko \cite{Fomenko19}. It was discovered that topological invariants of such systems often coincide with the invariants of systems of rigid body dynamics, mathematical physics or geometry.

In \cite{nikolaev2001non}, non-wandering flows with finitely many singular points on compact surfaces are classified up to a graph-equivalence by using a topological invariant, called a Conley-Lyapunov-Peixoto graph, which is equipped with rotation and weight functions.
In addition, it stated that an orbit complex (also called separatrix configuration) of a flow is also a complete invariant for the set of flows with ``finitely many  separatrices'' in the sense of Markus \cites{markus1954global,neumann1975classification,neumann1976global}.
These papers are referred more than one hundred papers, as mentioned in \cite{buendia2018markus}.
However, Buend{\'\i}a and L\'opez have pointed out that orbit complex is not a complete invariant for flows having  ``finitely many  separatrices''  in the sense of Markus (resp. finitely many singular points and no limit separatrices in the sense of Markus) \cite{buendia2018markus}.
In fact, one of their counterexamples is a toral flow which consists of one singular point and non-closed proper orbits.
In other words, they have shown that the Markus-Neumann theorem needs not work for the original setting using the simple counterexample.
However, it is shown that an orbit complex with additional some orders or a mild condition is complete for flows of finite type (e.g. Morse-Smale flows, Hamiltonian flows with finitely many singular points)  on compact surfaces \cite{yokoyama2017decompositions}.

To describe a complete invariant of Morse-Smale-like flows, we recall some concepts.
A periodic orbit is one-sided if it is either a boundary component of a surface or has a small neighborhood which is a M\"obius band.
For a Morse-Smale-like flow on a compact surface, denote by $\mathop{\mathrm{BD}}_{+}(v)$ the finite union of singular points, limit cycles, one-sided periodic orbits, and border separatrices.
Then $\mathop{\mathrm{BD}}_{+}(v)$ is the weakly border point set of $v$.
Notice that the definition $\mathop{\mathrm{BD}_+(v)}$ for a Morse-Smale-like flow is a reduction of one for a surface flow (see \cite{yokoyama2017decompositions}*{Lemma 7.1} for details).
Then we have the following observation for completeness.

\begin{main}\label{main:g}
The embedded directed graph $\mathop{\mathrm{BD}}_+(v)$  is a finite complete invariant of Morse-Smale-like flows on compact surfaces.
\end{main}

The observation implies that generic gradient flows on surfaces can be described by the finite complete invariants and so that any generic time evaluation of gradient flows on surfaces can be described by sequences of such finite invariants.

\subsection{Overview of the paper}
The present paper consists of ten sections.
In the next section, as preliminaries, we introduce the notions of combinatorics, topology, and dynamical systems.
In \S 3, we characterized isolated singular points of gradient flows on surfaces.
In particular, Theorem~\ref{main:a} is proved.
In \S 4, we topologically characterize gradient flows with finitely many singular points and Morse flows on compact surfaces (i.e. Theorem~\ref{main:c} and Theorem~\ref{main:Morse} are proved).
In \S 5, we introduce an operation by cutting limit circuits and collapsing the new boundary components into singletons to characterize a regular Morse-Smale-like flow.
We show that a quasi-regular flow on a compact surface is a Morse-Smale-like flow without elliptic sectors if and only if it is a quasi-Morse-Smale flow with finitely many singular points (i.e.  Theorem~\ref{main:qausi-MS} is proved).
In \S 6, we topologically characterize Morse-Smale flows on compact surfaces (i.e. Theorem~\ref{main:d} is proved).
In \S 7, we show the openness and density of Morse(-Smale) flows on compact surfaces (i.e. Theorem~\ref{th:open_dense_M} and Theorem~\ref{th:open_dense_MS} are proved).
In \S 8, we characterize and list generic intermediate Morse-Smale-like flows between gradient flows (i.e. Theorem~\ref{main:e} and Theorem~\ref{main:e} are proved).
In \S 9, we characterize and list all generic intermediate Morse-Smale-like flows between Morse-Smale flows (i.e. Theorem~\ref{th:generic_gradient_trans} and Theorem~\ref{th:generic_MS_trans} are proved).
In the final section, we observe that the embedded directed graph $\mathop{\mathrm{BD}}_+$ is a complete invariant of Morse-Smale-like flows on compact surfaces (i.e. Theorem~\ref{main:g} is proved).

\section{Preliminaries}

\subsection{Notion of dynamical systems}
By a surface,  we mean a two dimensional manifold, that does not need to be orientable.
A flow is a continuous $\R$-action on a manifold.
From now on, we suppose that flows are on surfaces.
Let $v : \R \times S \to S$ be a flow on a surface $S$.
For $t \in \R$, define $v_t : S \to S$ by $v_t := v(t, \cdot )$.
For a point $x$ of $S$, we denote by $O(x)$ the orbit of $x$.
A positive (resp. negative) orbit of $x$ is $v(\R_{>0}, x)$ (resp. $v(\R_{<0}, x)$), denoted by $O^+(x)$ (resp. $O^-(x)$).
Recall that a point $x$ of $S$ is singular if $x = v_t(x)$ for any $t \in \R$ and is periodic if there is a positive number $T > 0$ such that $x = v_T(x)$ and $x \neq v_t(x)$ for any $t \in (0, T)$.
An orbit is closed if it is singular or periodic.
An orbit is proper if it is embedded.
Denote by $\mathop{\mathrm{Sing}}(v)$ the set of singular points and by $\mathop{\mathrm{Per}}(v)$ (resp. $\mathop{\mathrm{Cl}}(v)$, $\mathrm{P}(v)$) the union of periodic (resp. closed, non-closed proper) orbits.
Denote by $\overline{A}$ the closure of a subset $A$.
Recall that the $\omega$-limit (resp. $\alpha$-limit)  set of a point $x$ is $\omega(x) := \bigcap_{n\in \mathbb{R}}\overline{\{v_t(x) \mid t > n\}}$ (resp.  $\alpha(x) := \bigcap_{n\in \mathbb{R}}\overline{\{v_t(x) \mid t < n\}}$), where the closure of a subset $A$ is denoted by $\overline{A}$.
A point is wandering if there are its neighborhood $U$ and a positive number $N$ such that $v_t(U) \cap U = \emptyset$ for any $t > N$.
A point is non-wandering if it is not wandering (i.e. for any its neighborhood $U$ and for any positive number $N$, there is a number $t \in \mathbb{R}$ with $|t| > N$ such that $v_t(U) \cap U \neq \emptyset$).
Denote by $\Omega (v)$ the set of non-wandering points, called the non-wandering set.
A point $x$ is recurrent if $x \in \alpha(x) \cup \omega(x)$.
Denote by $\mathrm{R}(v)$ the set of non-closed recurrent points.
Notice that a point is recurrent if and only if it is not proper, and so that $S = \mathop{\mathrm{Cl}}(v) \sqcup \mathrm{P}(v) \sqcup \mathrm{R}(v)$, $\sqcup$ denotes a disjoint union.
Moreover, the union $\mathrm{P}(v)$ corresponds to the set of non-recurrent points.
For an orbit $O$, define $\omega(O) := \omega(x)$ and $\alpha(O) := \alpha(x)$ for some point $x \in O$.
Note that an $\omega$-limit (resp. $\alpha$-limit) set of an orbit is independent of the choice of point in the orbit.
A subset is saturated (or invariant) if it is a union of orbits.
For a closed invariant set $\gamma$, define the stable subset $W^s(\gamma) := \{ y \in S \mid \omega(y) \subseteq  \gamma \}$ and the unstable subset  $W^u(\gamma) := \{ y \in S \mid \alpha(y) \subseteq \gamma \}$.
If the stable (resp. unstable) subset is immersed, then it is called the stable (resp. unstable) manifold.

\subsubsection{Types of singular points}

A multi-saddle is a singular point with finitely many separatrices, as in Figure~\ref{multi-saddles}.
A $k$-$\partial$-saddle (resp. $k$-saddle) is an isolated singular point on (resp. outside of) $\partial S$ with exactly $(2k + 2)$-separatrices, counted with multiplicity.
Note that a singular point is a multi-saddle if and only if it is a $k$-saddle or a $k/2$-$\partial$-saddle for some $k \in \mathbb{Z}_{\geq 0}$.
In other words, a singular point is a multi-saddle if and only if it is a finitely sectored singular point whose sectors are hyperbolic.
Similarly, a singular point is a sink, a $\partial$-sink, a source, or a $\partial$-source if and only if it is a finitely sectored singular point whose sectors consist of exactly one parabolic sector.

\subsubsection{Topological hyperbolicity}
Recall that a periodic orbit on a surface is topologically hyperbolic if and only if it is a limit cycle which is either semi-attracting on each side or semi-repelling on each side, and that a singular point on a surface is topologically hyperbolic if and only if it is a saddle, a $\partial$-saddle, a sink, a $\partial$-sink, a source, or a $\partial$-source.
In other words, a closed orbit is topologically hyperbolic if and only if it is locally topologically equivalent to a hyperbolic orbit with respect to the flow generated by a smooth vector field.

\subsubsection{Topological equivalence}
A flow $v$ on a topological space $X$ is topologically equivalent to a flow $w$ on a topological space $Y$ if there is a homeomorphism $h \colon X \to Y$ whose image of any orbit of $v$ is an orbit of $w$ and which preserves the direction of the orbits.

\subsubsection{Gradient flows}
A flow is gradient if it is topologically equivalent to a flow generated by a gradient vector field.

\subsubsection{Circuit}
A trivial circuit is a singular point.
Recall that an annular subset is homeomorphic to an annulus.
An open annular subset $\mathbb{A}$ of a surface is a collar of an invariant subset $\gamma$ if there is a \nbd $U$ of $\gamma$ such that $\mathbb{A}$ is a connected component of the difference $U - \gamma$.
Note that an open annular subset $\mathbb{A}$ of a surface is a collar of a singular point $x$ if and only if the union $\mathbb{A} \sqcup \{ x \}$ is a neighborhood of $x$.
By a cycle or a periodic circuit, we mean a periodic orbit.
A circuit is either a trivial or non-trivial circuit.
Note that there are non-trivial circuits with infinitely many edges and that any non-trivial non-periodic circuit contains non-recurrent orbits.
A circuit $\gamma$ is a semi-attracting (resp. semi-repelling) circuit with respect to a small collar $\A$ if $\omega(x)  = \gamma$ (resp. $\alpha(x) = \gamma$) and $O^+(x) \subset \A$ (resp. $O^-(x) \subset \A$) for any point $x \in \A$.
Then $\A$ is called a semi-attracting (resp. semi-repelling) collar basin of $\gamma$.
A non-trivial circuit $\gamma$ is a limit circuit if it is a semi-attracting or semi-repelling circuit.
Note that a trivial circuit is semi-attracting (resp. semi-repelling) if and only if it is either a $\partial$-source or a source (resp. a $\partial$-sink or a sink).
A limit circuit is (topologically) hyperbolic if each separatrix is contained in two limit circuits counted with multiplicity and is either semi-attracting on each side or semi-repelling on each side.
We call that a limit cycle (resp. circuit) is a fake limit cycle (resp. circuit) if it is semi-attracting on a side and semi-repelling on another side.
Note that a limit circuit is not hyperbolic if and only if it is fake.
Moreover, the non-existence of fake limit cycles is necessary to fix the number of limit cycles under perturbations.
In fact, a pair of attracting limit cycle and a repelling limit cycle can be merged into a fake limit cycle $C$ as in Figure~\ref{fake_lc} which can be annihilated by arbitrarily small perturbations.
\begin{figure}
\begin{center}
\includegraphics[scale=0.285]{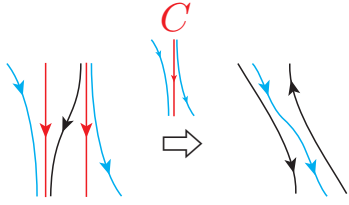}
\end{center}
\caption{A pair annihilation of an attracting limit cycle and a repelling limit cycle.}
\label{fake_lc}
\end{figure}%

\subsubsection{Flow box}
A non-degenerate arc is an orbit arc if it is contained in an orbit.
A subset $U$ is a flow box if there are bounded non-degenerate intervals $I, J$ and a homeomorphism $h \colon I \times J \to U$ such that the image $h(I \times \{ t \})$ for any $t \in J$ is an orbit arc as in Figure~\ref{flow-box}.
\begin{figure}
\begin{center}
\includegraphics[scale=0.4]{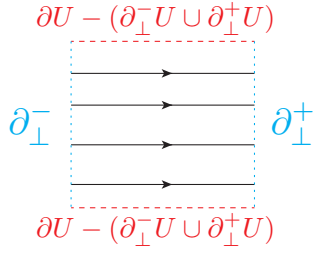}
\end{center}
\caption{An open flow box $U$ and a decomposition of its boundary}
\label{flow-box}
\end{figure}
A flow box is open if it is open.

\subsection{Some concepts on metric spaces}

Recall that the Hausdorff distance between nonempty subsets of a metric space as follows:
Let $(M, d)$ be a metric space.
Denote by $B_r(x)$ an open ball centered at a point $x \in M$ with radius $r>0$ (i.e. $B_r(x) := \{ y \in M \mid d(x,y) < r \}$).
Put $B_r(A) := \bigcup_{a \in A}B_r(a)$ for a subset $A \subseteq M$.
For a pair of nonempty subsets $A$ and $C$ of a metric space with the metric $d$, define the Hausdorff distance $d_H(A, C)$ between $A$ and $C$ by
$d_H(A,C) := \inf \{ \varepsilon > 0 \mid A \subseteq B_\varepsilon (C), C \subseteq B_\varepsilon (A) \}$.

\section{Characterizations of isolated singular points of gradient flows on surfaces}

We state the existence of a loop that is transverse at all points but finitely many points.

\begin{lemma}\label{almost_closed_transverse}
Let $v$ be a flow on a surface $S$.
Suppose that each singular point is isolated.
For any loop $\gamma'$, there is a loop $\gamma$ which is arbitrarily close to $\gamma$ with respect to the Hausdorff distance such that $\gamma$ is a loop which is smooth and which is transverse except finitely many tangencies.
\end{lemma}

\begin{proof}
By Gutierrez's smoothing theorem~\cite{gutierrez1978structural}, the flow $v$ is topologically equivalent to a $C^1$-flow and so is generated by a continuous vector field on $S$.
We may assume that $S$ is connected and $v$ is generated by a continuous vector field $X_v$.
Since $\mathop{\mathrm{Sing}}(v)$ is closed and zero-dimensional, the complement $S - \mathop{\mathrm{Sing}}(v)$ is a connected surface with $\overline{S - \mathop{\mathrm{Sing}}(v)} = S$.
Fix a loop $\gamma'$.
Let $d_H$ be a Hausdorff distance induced by a Riemannian metric $d$ of $S$ and $\varepsilon > 0$ a small positive number.
Since $\gamma'$ is compact, the intersection $\mathop{\mathrm{Sing}}(v) \cap \gamma'$ is finite.
By arbitrarily small perturbation, from the open denseness of $S - \mathop{\mathrm{Sing}}(v)$,  isolatedness of singular points implies that there is a loop $\gamma_0 \subset  B_{\varepsilon/4}(\gamma')$ with $d_H(\gamma', \gamma_0) < \varepsilon/4$ and $\gamma_0 \cap \mathop{\mathrm{Sing}}(v) = \emptyset$.
Since $\gamma_0 \cap \partial B_{\varepsilon/4}(\gamma') = \emptyset$, compactness of $\gamma_0$ and the boundary $\partial B_{\varepsilon/4}(\gamma')$ implies that the minimum $e:= \min \{ d(x,b) \mid x \in \gamma_0, b \in \partial B_{\varepsilon/4}(\gamma') \}$ is positive.
Put $e_1 := \min \{ e, \varepsilon/4 \}$.
Then $d_H(A, B) \leq \varepsilon/2$ for any point $x$ and any nonempty subsets $A$ and $B$ of $B_{e_1}(x)$.

We claim that for any point $c \in \gamma_0$, there is an open flow box $B$ with $c \in B$ whose intersection $B \cap \gamma_0$ is an open arc such that for any pair of points $a \neq b \in B \cap \gamma_0$ there is an open arc between them in $B$ which either is an orbit arc or is transverse.
Indeed,
if $\gamma_0$ is transverse at $c$, then there is an open flow box $B$ whose intersection $B \cap \gamma_0$ is an open transverse arc.
Thus we may assume that $\gamma_0$ is tangent at $c$.
Fix a direction on $\gamma_0$.
Denote by $(a,b)_{\gamma_0}$ the open sub-arc of $\gamma_0$ from a point $a \in \gamma_0$ to a point $b \in \gamma_0$.
The smoothness of $\gamma_0$ implies that we can parametrize $\gamma_0$ by $t \in (-\varepsilon, \varepsilon)$ near $c$ such that $g(X_v(\gamma_0(t)), \frac{d}{dt}\gamma_0(t)) >0$ near $c$, where $\varepsilon>0$ is some small positive number and $g$ is a Riemannian metric.
This means that $\gamma_0$ never return with respect to the flow direction.
Therefore there is an open flow box $B$ whose intersection $B \cap \gamma_0$ is an open arc with $(B \cap \gamma_0) \cap \partial_\pitchfork B = \emptyset$ as in Figure~\ref{flowbox}.
\begin{figure}
\begin{center}
\includegraphics[scale=0.4]{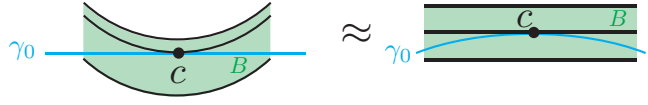}
\end{center}
\caption{An open flow box $B$ with $c \in B \cap \gamma_0$ whose intersection $B \cap \gamma_0$ is an open arc with $(B \cap \gamma_0) \cap \partial_\pitchfork B = \emptyset$.}
\label{flowbox}
\end{figure}
Moreover, for any pair of points $a \neq b \in B \cap \gamma_0$, there is an open transverse arc between $a$ and $b$ unless there is an open obit arc in $B$ between $a$ and $b$.
This completes the claim.

Therefore, for any point $x \in \gamma_0$, there is such an open flow box $B_x \subset B_{e_1}(x)$.
Then $\gamma_0 \subset B_{e_1}(\gamma_0) \subseteq B_{\varepsilon/4}(\gamma')$.
By compactness of $\gamma_0$, there are finitely many such flow boxes $B_1, \ldots , B_k$ such that $\gamma_0 \subset \bigcup_{i=1}^k B_i \subseteq B_{\varepsilon/4}(\gamma')$.
Fix a point $x_i$ with $x_i \in B_i = B_{e_i}(x)$.
Since $B_i \cap \gamma_0$ is an open directed arc, denote by $\omega_i$ (resp. $\alpha_i$) the positive (resp. negative) boundary of $B_i \cap \gamma_0$.
Put $j_1 := 1$.
By $\gamma_0 \subset \bigcup_{i=1}^k B_i$, there are an integer $i_1 = 1, \ldots , k$ with $i_1 \neq j_1$ and a small open \nbd $U_{j_1} \subset B_{i_1}$ of $\omega_{j_1}$.
Since $\omega_{j_1} \in B_{j_1} \cap \gamma_0$, fix a point $w_1 \in U_{j_1} \cap B_{j_1} \cap \gamma_0 \subset B_{i_1} \cap B_{j_1} \cap \gamma_0$.
By induction for $j_l$, define $w_l$ as follows:
Suppose that we have an open flow box $B_{i_{l-1}}$ and a small open \nbd $U_{j_{l-1}} \subset B_{i_{l-1}}$ of $\omega_{j_{l-1}}$ with $w_{l-1} \in U_{j_{l-1}} \cap B_{j_{l-1}} \cap \gamma_0 \subset B_{i_{l-1}} \cap B_{j_{l-1}} \cap \gamma_0$.
Put $j_l := i_{l-1}$.
By $\gamma_0 \subset \bigcup_{i=1}^k B_i$, there are an integer $i_l = 1, \ldots , k$ with $i_l \neq j_l$ and a small open \nbd $U_{j_l} \subset B_{i_l}$ of $\omega_{j_l}$.
Since $\omega_{j_l} \in B_{j_l} \cap \gamma_0$, fix a point $w_l \in U_{j_l} \cap B_{j_l} \cap \gamma_0 \subset B_{i_l} \cap B_{j_l} \cap \gamma_0$.

We claim that there is a natural number $n$ such that $\gamma_0 = (\alpha_{j_1}, \omega_{j_1})_{\gamma_0} \cup (\bigcup_{m=1}^n\{\omega_{j_m} \} \cup (\omega_{j_m}, \omega_{j_{m+1}})_{\gamma_0} )$.
Indeed, assume that there are no such natural numbers.
Then $\gamma_0 \cap B_{j_1} \subseteq (\alpha_{j_1}, \omega_{j_l})_{\gamma_0}$ and so $B_{j_1} \neq B_{j_l}$ for any $l \in \Z_{> 2}$.
Similarly, since $\gamma_0 \cap \bigcup_{m=1}^l B_{j_l} \subseteq (\alpha_{j_1}, \omega_{j_1})_{\gamma_0} \cup (\bigcup_{m=1}^l \{\omega_{j_m} \} \cup (\omega_{j_m}, \omega_{j_{m+1}})_{\gamma_0} ) = (\alpha_{j_1}, \omega_{j_{l+1}})_{\gamma_0}$, we have $B_{j_l} \neq B_{j_m}$ for any $l \in \Z_{> 2}$ and any $m \in \Z_{> l + 1}$, which contradicts that there are at most $k$ such open flow boxes $B_{j_m}$.

Therefore there are $w_1, \ldots , w_n$ with $\gamma_0 = (w_{n}, w_1)_{\gamma_0} \cup \bigcup_{m =1}^{n-1} (w_{m}, w_{m+1})_{\gamma_0}$ such that $(w_1, w_2)_{\gamma_0} \subseteq \gamma_0 \cap B_{i_2}$, \ldots , $(w_{n-1}, w_n)_{\gamma_0} \subseteq \gamma_0 \cap B_{i_n}$, and $(w_n, w_1)_{\gamma_0} \subseteq \gamma_0 \cap B_{i_1}$.
Then $w_2 \in B_{j_1} \cap B_{j_2}$, \ldots, $w_{n} \in B_{j_{n-1}} \cap B_{j_n}$,
and $w_1 \in B_{j_n} \cap B_{j_1}$.
Hence we can choose closed arcs $[w_m, w_{m+1}]_{\gamma_1}$ for any $m = 1, \ldots , n-1$ (resp $[w_{n}, w_1]_{\gamma_1}$) from $w_m$ (resp. $w_n$) to $w_{m+1}$ (resp. $w_1$) in $B_{j_{m+1}}$ (resp. $B_{j_1}$) which either is an orbit arc or is transverse.
Then the union $\gamma_1 := [w_{n}, w_1]_{\gamma_1} \cup \bigcup_{m=1}^{n-1} [w_m, w_{m+1}]_{\gamma_1}$ is a loop consists of finitely many closed orbit arcs and closed transverse arcs.
By $(w_m, w_{m+1})_{\gamma_0}, [w_m, w_{m+1}]_{\gamma_1} \subset B_{j_m} = B_{e_1}(x_{j_m}) \subseteq B_{\varepsilon/4}(x_{j_m})$, we obtain $d_H((w_m, w_{m+1})_{\gamma_0}, [w_m, w_{m+1}]_{\gamma_1}) \leq \varepsilon/2$ and so $d_H(\gamma_0, \gamma_1) \leq \varepsilon/2$.
Since any closed orbit arcs can be approximated by a finite union of closed transverse arcs, the loop $\gamma_1$ can be approximated by a loop $\gamma_2$ consisting of closed transverse arcs such that $d_H(\gamma_1 , \gamma_2) < \varepsilon/4$.
Then $\gamma_2$ is a loop that is piecewise smooth and which is transverse at all points but finitely many non-smooth points.
By arbitrarily small perturbation to $\gamma_2$, we may assume that $\gamma_2$ is smooth and transverse except finitely many tangencies with respect to $w$.
This implies that $d_H(\gamma', \gamma_2) < \varepsilon$ and the resulting loop $\gamma$ is desired.
\end{proof}

We describe saddle-type singular points of flows on surfaces.

\begin{lemma}\label{lem:almost_transverse}
Let $v$ be a flow on a surface, $x$ an isolated singular point which is not a topological center, $\gamma$ a simple closed curve which is transverse except finitely many tangencies, $B$ a closed disk with $\partial B = \gamma$ and $x \in \mathrm{int} B$.
Suppose that, for any point $y \in B$, the connected component of $O(y) \cap B$ containing $y$ is a closed orbit arc $C_y$ whose end points $\partial C_y$ are two distinct points in $\{ x \} \sqcup \partial B$.
Then $x$ is a finitely sectored singular point without elliptic sectors.
\end{lemma}

\begin{proof}
Suppose that $\gamma$ is a closed transversal.
Then either $\omega(y) = x$ for any point $y \in B$, or $\alpha(y) = x$ for any point $y \in B$.
This means that $x$ is either a sink or a source and so that $x$ is a finitely sectored singular point without elliptic sectors.

Thus we may assume that $\gamma$ is not a closed transversal.
Let $I$ be the union of intersections of $\gamma$ and orbits whose either $\omega$-limit limit sets are $x$ and positive orbits are contained in $B$ or $\alpha$-limit sets are $x$ and negative  orbits are contained in $B$.

We claim that the complement $\gamma - I$ is open in $\gamma$.
Indeed, we may assume that $\gamma \neq I$.
Fix a point $y \in I$.
For any point $w \in B = \gamma \sqcup \mathrm{int} B$, denote by $C_w$ the connected component of $O(w) \cap B$ containing $w$.
Since $x$ is an isolated singular point,  the connected component $C_y$ is an orbit arc connecting $\gamma$ and so compact.
By the flow box theorem for $C_y$ (cf. Theorem 1.1, p.45\cite{aranson1996introduction}), the complement $\gamma - I$ is open in $\gamma$.

By the finite existence of tangencies, the complement $\gamma - I$ consists of finitely many open intervals and so that $I$ consists of finitely many closed intervals.
By hypothesis, each point whose $\omega$- or $\alpha$-limit set is $x$ intersects $\gamma$.
This means that the union $\bigcup_{w \in I} \mathrm{int} C_w$ consists of finitely many parabolic sectors.
Therefore it suffices to show that the closure $\overline{\bigcup_{w \in \gamma - I}C_w}$ contains a finite union $H$ of hyperbolic sectors such that the union $(\{ x \} \sqcup H) \cup \bigcup_{w \in I}C_w$ is a \nbd of $x$.

Fix a connected component $\mu$ of $\gamma - I$.
Then $\mu$ is an open interval with the boundary $\{ y, z \} = \overline{\mu} \cap I$ for some $y, z \in I$.
By a perturbation of $\mu \subset \gamma$ as in Figure~\ref{elim_tangency} if necessary, we may assume that $C_y$ and $C_z$ have no tangencies.
\begin{figure}
\begin{center}
\includegraphics[scale=0.4]{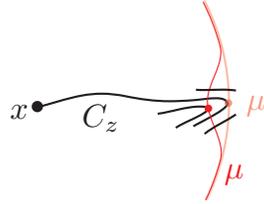}
\end{center}
\caption{A perturbation of $\mu \subset \gamma$.}
\label{elim_tangency}
\end{figure}
Let $B_\mu$ be the connected component of $B - \bigcup_{w \in I} C_w$ containing $\mu$.
Then $\partial B_\mu = \{ x \} \sqcup C_y \sqcup C_z \sqcup \mu$.
Since the boundary of any orbit arc $C_w$ for any $w \in \mu$ consists of two points of $\mu$, by a perturbation of $\mu \subset \gamma$ as in Figure~\ref{elim_inner_tangency} if necessary, we may assume that there are no inner tangencies.
\begin{figure}
\begin{center}
\includegraphics[scale=0.4]{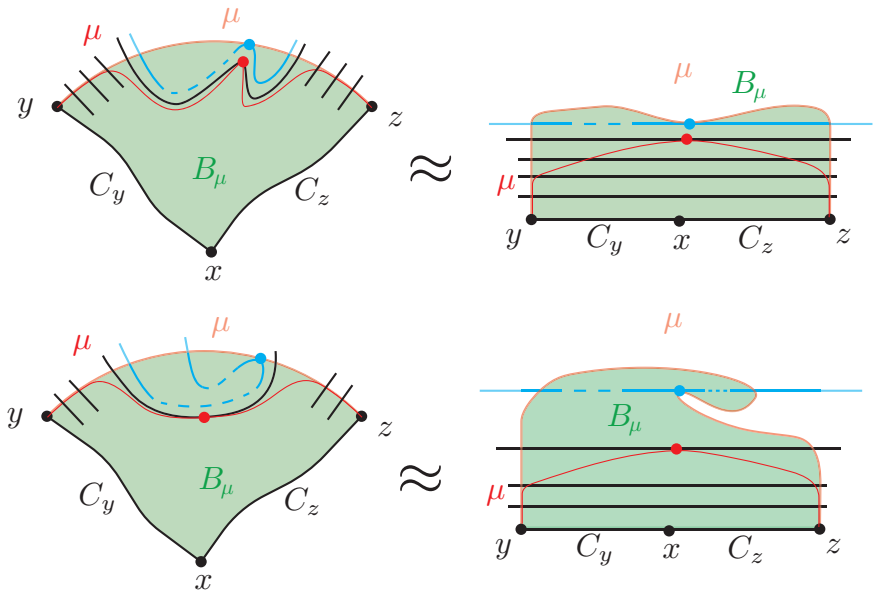}
\end{center}
\caption{Deformations of sectors to eliminate inner tangencies.}
\label{elim_inner_tangency}
\end{figure}
Thus $\mu$ contains only outer tangencies.

We claim that $\mu$ contains exactly one tangency.
Indeed, assume that $\mu$ contains no tangencies.
Then $\overline{\mu} = \mu \sqcup \{y,z \}$ is a closed transverse arc and so either $\{ x \} = \omega(y) = \omega(z)$ or $\{ x \} = \alpha(y) = \alpha(z)$.
By time reversing if necessary, we may assume that $\{ x \} = \omega(y) = \omega(z)$.
Since $\partial B_\mu = \{ x \} \sqcup C_y \sqcup C_z \sqcup \mu$, the positive orbit $O^+(w)$ for any point $w \in \mu$ is contained in $C_w$.
Since $\omega(w)$ is a singular point and $C_w \cap \Sv = \{ x \}$, we have $\omega(w) = \{ x \}$ and so $w \in I \cap \mu = \emptyset$, which is a contradiction.
Thus there are tangencies on $\mu$.
Assume that there are at least two tangencies $t_1$ and $t_2$.
Fix an orientation of $\mu$ with $y < z$.
We may assume that there are no tangencies on $\mu_{t_1, t_2}$ and $y < t_1 < t_2 < z$ as in Figure~\ref{two_tang}.
Denote by $\mu_{t,s}$ the open arc contained in $\mu$ between $t < s \in \mu$.
\begin{figure}
\begin{center}
\includegraphics[scale=0.4]{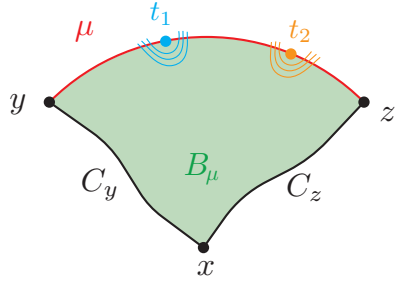}
\end{center}
\caption{Two successive tangencies on $\mu$.}
\label{two_tang}
\end{figure}
For $t \in \mu_{t_1, t_2}$, denote by $t'$ another boundary point of $C_t$ (i.e. $\{t, t' \} = \partial C_t$).
Since $t_i$ are outer tangencies, for any $t \in \mu_{t_1, t_2}$ near $t_1$ (resp. $t_2$), we have $t' \in \mu_{y,t_1}$ (resp. $t' \in \mu_{t_2,z}$).
Put $\mu_1 := \{ t \in \mu_{t_1, t_2} \mid t' \in \mu_{y,t_1}\}$ and $\mu_2 := \{ t \in \mu_{t_1, t_2} \mid t' \in \mu_{t_2,z}\}$
By the flow box theorem for $C_t$, both $\mu_i$ are open subsets.
Therefore $\mu_{t_1, t_2} \subset \mu \subset \gamma -I$ is an open interval which is covered by open subsets $\mu_1$ and $\mu_2$ and so is not connected, which contradicts that any open intervals on a plane are connected.

Let $t_0 \in \mu$ be the unique tangency contained in $\mu$.
For any point $t \in \mu - \{ t_0 \}$ and for any small non-degenerate closed interval $\mu_t$ in $\mu$ containing $t$, the union $\bigcup_{s \in \mu_t} C_s$ is a trivial closed flow box.
This means that the interior $\mathrm{int} B$ is a trivial open flow box and so the union $C_y \sqcup C_z \sqcup \mathrm{int} B$ is a hyperbolic sector.

For any connected component of $\gamma - I$, applying the perturbations finitely many times, we can obtain a simple closed curve $\gamma_1 \subset B$ which bounds an open disk that contains $x$ and consists of finitely many hyperbolic or parabolic sectors.
\end{proof}

\subsection{Proof of Theorem~\ref{main:a}}
%
Let $x$ be an isolated singular point of a gradient flow $v$ on a surface.
Since $v$ is gradient, a generalization of the Poincar\'e-Bendixson theorem for a flow with finitely many singular points (cf. Theorem 2.6.1~\cite{nikolaev1999flows}) implies that each of $\omega$-limit set and $\alpha$-limit set of a point is a singular point of $v$.
The gradient property implies that $\omega$-limit set and $\alpha$-limit set of any non-singular point are different from each other.
Since a gradient flow contains no centers, the singular point $x$ is not a center.
By Lemma~\ref{almost_closed_transverse}, there is a loop $\gamma$ which bounds a closed disk $B$ and which is transverse except finitely many tangencies such that $x \in \mathrm{int} B$ and $B \cap \Sv = \{ x \}$.
Since $v$ is gradient, for any point $y \in B$, the connected component of $O(y) \cap B$ containing $y$ is a closed arc such that the end points of the closed arc $C_y$ are two distinct points in $\{ x \} \sqcup \partial B$.
Applying Lemma~\ref{lem:almost_transverse} to $x$, the isolated singular point $x$ is a finitely sectored singular point without elliptic sectors.

\subsection{Isolated singular points in quasi-regular gradient flows}
We show the following equivalence.

\begin{corollary}\label{Morse01}
A gradient flow with finitely many singular points on a compact surface is quasi-regular if and only if no singular points have both a hyperbolic sector and a parabolic sector.
\end{corollary}

\begin{proof}
Let $v$ be gradient flow on a compact surface.
Theorem~\ref{main:a} implies that each singular point is a non-trivial finitely sectored singular point without elliptic sectors.
Since a non-elliptic sector is hyperbolic or parabolic, the assertion holds.
\end{proof}

\section{Topological characterization of gradient flows with finitely many singular points and of Morse flows}

In this section, we characterize gradient flows with finitely many singular points and Morse flows.

\subsection{Proof of Theorem~\ref{main:c}}

By definitions, there are at most finitely many singular points in any case.
We may assume hat $S$ is connected.

Suppose that $v$ is a gradient flow.
Then there are neither cycles nor elliptic sectors.
Since gradient flows has no non-closed recurrent orbits, we have $S = \mathop{\mathrm{Sing}}(v) \sqcup  \mathrm{P}(v)$.
Theorem~\ref{main:a} implies that each singular point is a non-trivial finitely sectored singular point without elliptic sectors.
This means that $v$ is Morse-Smale-like without elliptic sectors or non-trivial circuits.

Suppose that $v$ is a Morse-Smale-like flow without elliptic sectors or non-trivial circuits.
By definition, we obtain $S = \mathop{\mathrm{Sing}}(v) \sqcup \mathop{\mathrm{Per}}(v) \sqcup \mathrm{P}(v)$.
The density of $\mathrm{P}(v)$ implies that $\mathop{\mathrm{Per}}(v)$ has no interior.
By \cite{yokoyama2017decompositions}*{Lemma~3.4}, each periodic orbit is a limit cycle, and so there are no periodic orbits.
This means that $S = \mathop{\mathrm{Sing}}(v) \sqcup \mathrm{P}(v)$.
Since there are no centers, the hypothesis implies that each singular point consists of  either hyperbolic or parabolic sectors.
Let $U$ be a connected component of the complement of $\mathop{\mathrm{BD}}_+(v)$.
By \cite{yokoyama2017decompositions}*{Theorem~7.7}, the component $U$ is either an open transverse annulus or an open invariant flow box.
A generalization of the Poincar\'e-Bendixson theorem for a flow with finitely many singular points implies that each of the $\omega$-limit set and the $\alpha$-limit set of a point is a singular point.
Therefore if $U$ is an open transverse annulus, then the closure is a sphere and the flow is a gradient flow with exactly one sink and one source.
The non-existence of elliptic sectors implies that the $\omega$-limit set and the $\alpha$-limit set of a point in $U$ are distinct.
Thus we may assume that each connected component of the complement of $\mathop{\mathrm{BD}}_+(v)$ is such an open invariant flow box.
By $S = \mathop{\mathrm{Sing}}(v) \sqcup \mathrm{P}(v)$, since the border set $\mathop{\mathrm{BD}}_+(v)$ is a directed graph without no non-trivial circuits, the border set $\mathop{\mathrm{BD}}_+(v)$ has a smooth height function on $S$ and can be extended into the surface because the set difference $S - \mathop{\mathrm{BD}}_+(v)$ is a finite disjoint union of invariant flow boxes.
Hence $v$ is a gradient.

%
%

\subsection{Characterization of Morse flows}

Recall Morse-Smale flows on compact surfaces.

\subsubsection{Morse-Smale flows on compact manifolds}
A $C^r$ vector field $X$ for any $r \in \Z_{\geq0}$ on a closed manifold is Morse-Smale if $(1)$ the non-wandering set $\Omega(X)$ consists of finitely many hyperbolic closed orbits; $(2)$ any point in the intersection of the stable and unstable manifolds of closed orbits are transversal (i.e $W^s(O) \pitchfork W^u(O')$ for any orbits $O, O' \subset \Omega(X)$, where $W^s(O)$ is the stable manifold of $O$ and $W^u(O')$ is the unstable manifold of $O'$).
Here the transversality of submanifolds $A$ and $B$ on a manifold $M$ means that the submanifolds $A$ and $B$ span the tangent spaces for $M$ (i.e. $T_{A\cap B} M = T_{A\cap B} A + T_{A\cap B} B$).
Similarly, under two generic conditions for differentials to guarantee structural stability, Labarca and Pacifico defined a Morse-Smale vector field on a compact manifold as follows \cite{labarca1990stability}.
A $C^\infty$ vector field $X$ on a compact manifold is Morse-Smale if it satisfies the following conditions: \\
$(\mathrm{MS}1)$ The non-wandering set $\Omega(X)$ consists of finitely many hyperbolic closed orbits;
\\
$(\mathrm{MS}2)$ The restriction $X|_{\partial M}$ is a Morse-Smale vector field on a closed manifold;
\\
$(\mathrm{MS}3)$ For any orbits $O, O' \subset \Omega(X)$ and for any non-transversal point $x \in W^s(O) \cap W^u(O')$, we have $x \in \partial M$ and $(O \cup O') \cap \mathop{\mathrm{Sing}}(X) \neq \emptyset$.
\\
Therefore we say that a flow is Morse-Smale if it is topologically equivalent to a flow generated by a vector field satisfying the conditions $(\mathrm{MS}1)$--$(\mathrm{MS}3)$.
A flow is Morse if it is a Morse-Smale flow without limit cycles.
Note that the two generic conditions for differentials form an open dense subset of the set of $C^\infty$ vector fields and that they are stated as follows: any closed orbit is $C^2$ linearizable, and the weakest contraction (resp. expansion) at any closed orbit is defined.
Here the weakest contraction at a singular (resp. periodic) point $p$ is defined if the contractive eigenvalue with the biggest real part among the contractive eigenvalues of $DX(p)$ (resp. $DX_f(p)$, where $X_f$ is  the Poincar\'e map) is simple. Dually we can define that the weakest expansion at $p$ is defined.

Palis and Smale showed that a Morse-Smale $C^r$ vector field on a closed manifold is structurally stable with respect to the set of $C^r$ vector fields~\cites{palis1969morse,palis2000structural}.
Similarly, the assertion also holds for Morse-Smale vector fields on compact manifolds under the $C^2$ linearizable condition and the eigenvalue conditions \cite{labarca1990stability}.
%
%
We have the following observation.

\begin{lemma}\label{MS_char}
A flow $v$ is a Morse-Smale flow on a compact surface $S$ if and only if it satisfies the following conditions:
\\
$(\mathrm{M}0)$ $\Omega(v) = \mathop{\mathrm{Cl}}(v)$.
\\
$(\mathrm{M}1)$ The non-wandering $\Omega(v)$ set consists of finitely many orbits.
\\
$(\mathrm{M}2)$ Each singular point is either a sink, a $\partial$-sink, a source, a $\partial$-source, a saddle or a $\partial$-saddle.
\\
$(\mathrm{M}3)$ Each periodic orbit is a topologically hyperbolic limit cycle.
\\
$(\mathrm{M}4)$ Each multi-saddle separatrix is contained in the boundary $\partial S$.
\end{lemma}

The previous lemma implies the following observation.

\begin{lemma}\label{grad_MS_char}
A gradient flow with finitely many singular points satisfies the conditions $(\mathrm{M}0)$, $(\mathrm{M}1)$, and $(\mathrm{M}3)$.
\end{lemma}

\subsubsection{Topological characterization of Morse flows on compact surfaces}

We have the following statement.

\begin{lemma}\label{Morse}
The following are equivalent for a gradient flow $v$ with finitely many singular points on a compact surface:
\\
$(1)$ The flow $v$ is Morse.
\\
$(2)$ The flow $v$ is regular and each multi-saddle separatrix is contained in the boundary $\partial S$.
\\
$(3)$ Each singular point with hyperbolic sectors is either a saddle or a $\partial$-saddle and each multi-saddle separatrix is contained in the boundary $\partial S$.
\end{lemma}

\begin{proof}
In any case, each multi-saddle separatrix is contained in the boundary of a surface, the condition $(\mathrm{M}4)$ is satisfied,  and there are no limit cycles.
Let $v$ be gradient flow on a compact surface $S$ such that each multi-saddle separatrix is contained in the boundary $\partial S$.
Theorem~\ref{main:a} implies that each singular point of $v$ is a non-trivial finitely sectored singular point without elliptic sectors.
Lemma~\ref{MS_char} implies that $v$ satisfies the conditions $(\mathrm{M}0)$, $(\mathrm{M}1)$, and $(\mathrm{M}3)$.

Suppose that $v$ is Morse.
Then each singular point is either a ($\partial$-)sink, a ($\partial$-)source, or a ($\partial$-)saddle.
This means that $v$ is regular.
Therefore the condition {\rm(1)} implies the condition {\rm(2)}.

Suppose that $v$ is regular.
Then there are at most finitely many singular points and each singular point is either a ($\partial$-)sink, a ($\partial$-)source, or a ($\partial$-)saddle.
This implies that each singular point with hyperbolic sectors is either a saddle or a $\partial$-saddle.

Suppose that each singular point with hyperbolic sectors is either a saddle or a $\partial$-saddle.
Then the condition $(\mathrm{M}4)$ holds.
This implies $v$ is Morse.
\end{proof}

Lemma~\ref{Morse} and Theorem~\ref{main:c} imply Theorem~\ref{main:Morse}

\section{Quasi-Morse-Smale flows}\label{qms_flow}

In this section, the equivalence for Morse-Smale-like property and quasi-Morse-Smale property for quasi-regular flows with finitely many limit cycles.
First, we have the following properties of $v_{\mathrm{col}}$.

\begin{lemma}\label{lem:col}
The following statements hold for a flow $v$ on a compact surface $S$:
\\
$(1)$ $S - \Gamma = S_{\mathrm{me}} - \Gamma_{\mathrm{me}} = S_{\mathrm{col}} - \Gamma_{\mathrm{col}}$.
\\
$(2)$ $v = v_{\mathrm{me}} = v_{\mathrm{col}}$ on the difference $S - \Gamma$.
\\
$(3)$ $\Gamma \subset  \mathop{\mathrm{Sing}}(v) \sqcup \mathrm{P}(v)$, $\Gamma_{\mathrm{me}} \subset  \mathop{\mathrm{Sing}}(v_{\mathrm{me}}) \sqcup \mathrm{P}(v_{\mathrm{me}})$, $\Gamma_{\mathrm{col}} \subset  \mathop{\mathrm{Sing}}(v_{\mathrm{col}}) \sqcup \mathrm{P}(v_{\mathrm{col}})$.
\\
$(4)$ $\mathrm{R}(v) = \emptyset$ if and only if $\mathrm{R}(v_{\mathrm{col}}) = \emptyset$.
\\
$(5)$ The flow $v$ has no singular points with elliptic sectors if and only if so does $v_{\mathrm{col}}$.
\end{lemma}

\begin{proof}
By construction, we have that $S - \Gamma = S_{\mathrm{me}} - \Gamma_{\mathrm{me}} = S_{\mathrm{col}} - \Gamma_{\mathrm{col}}$ and so that $v = v_{\mathrm{me}} = v_{\mathrm{col}}$ on the difference $S - \Gamma$.
Since $\Gamma \subseteq \mathop{\mathrm{Sing}}(v) \sqcup \mathrm{P}(v)$, we obtain  $\Gamma_{\mathrm{me}} \subseteq \mathop{\mathrm{Sing}}(v_{\mathrm{me}}) \sqcup \mathrm{P}(v_{\mathrm{me}})$ and $\Gamma_{\mathrm{col}} \subseteq \mathop{\mathrm{Sing}}(v_{\mathrm{col}}) \sqcup \mathrm{P}(v_{\mathrm{col}})$.
 Then $\mathrm{R}(v) \subseteq S - \Gamma =  S_{\mathrm{col}} - \Gamma_{\mathrm{col}}$.
Since $\mathrm{R}(v) \subseteq S - \Gamma$, the assertion $(2)$ implies the assertion $(4)$.
Since each non-trivial circuit with non-trivial holonomy has only hyperbolic sectors on the sides with non-trivial holonomy and each new collapsed singular point is either a sink or a source and so has a parabolic sector, any elliptic sectors are invariant under the deformation to obtain $v_{\mathrm{col}}$.
This means that $v$ has no singular points with elliptic sectors if and only if so does $v_{\mathrm{col}}$.
\end{proof}


Notice that a circuit with non-orientable holonomy is appeared as a limit of Morse flows (see Figure~\ref{nonori_hol}).
\begin{figure}
\begin{center}
\includegraphics[scale=0.285]{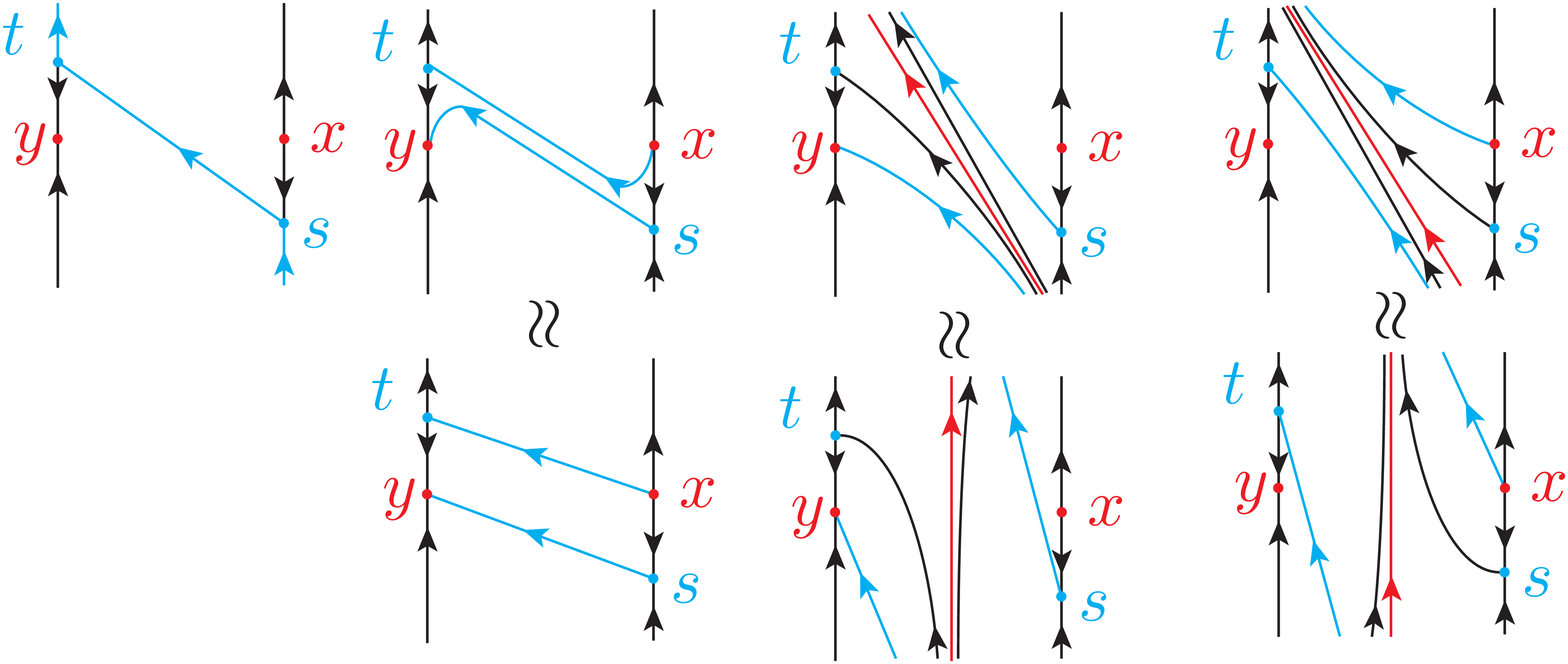}
\end{center}
\caption{All flows in the figure are flows on a M\"obius band with a $\partial$-sink, a $\partial$-source, two $\partial$-saddles, and non-recurrent orbits.
The flow in the figure on the left is not Morse-Smale but regular Morse-like and quasi-Morse-Smale (see definition in \S \ref{qms_flow}) and has a non-trivial circuit with non-limit non-orientable holonomy.
The flow in the second figure from the left is Morse. The flows in the second (resp. first) figures from the right are Morse-Smale and have a repelling (resp. attracting) limit cycle.}
\label{nonori_hol}
\end{figure}%
We have the following dichotomy of non-trivial circuits with non-trivial holonomy.

\begin{lemma}\label{lem:dich}
Each non-trivial circuit with non-trivial holonomy for a flow with finitely sectored singular points and with at most finitely many limit cycles on a surface either has non-orientable holonomy or is a limit circuit.
\end{lemma}

\begin{proof}
Let $v$ be a flow with finitely sectored singular points and with at most finitely many limit cycles on a surface $S$ and $\gamma$ a non-trivial circuit with non-trivial holonomy.

We claim that we may assume that there is a small open annular \nbd $A$ of the circuit $\gamma$ such that $v|_A$ is the restriction of a spherical flow.
Indeed, considering a small open annular \nbd $A$ of the circuit $\gamma$ and collapsing the two boundary components of $\partial A$, the resulting surface is a sphere.
This means that $v|_A$ can be considered as the restriction of a spherical flow.

The flow box theorem for a continuous flow on a compact surface implies that there are a point $y \in \gamma$ and small open transverse arcs $I$ and $J$ with $y \in \partial I \cap \partial J$ such that either $I \subseteq J$ or $J \subseteq I$ and that the first return map $f_v: I \to J$ is well-defined.
The existence of holonomy implies that there is an open flow box $U$ whose positive (resp. negative) vertical boundary $\partial_\perp^- U$ (resp. $\partial_\perp^+ U$) are $I$ (resp. $J$) such that $\gamma$ is a connected component of the difference $\partial U - (\partial_\perp^- U \cup \partial_\perp^+ U)$.
Then $U$ contains only hyperbolic sectors if sectors exists in $U$.
If $f_v$ is not orientation-preserving, then it is orientation-reversing and so has non-orientable holonomy.
Thus we may assume that $f_v$ is orientation-preserving.

We claim that $\gamma$ is a limit circuit.
Indeed, assume that $\gamma$ is not a limit circuit.
Then there is a sequence $(x_n)$ in $I$ to converge to a non-singular point $x \in \gamma$ such that each point $x_n$ is a fixed point of $f_v$.
Non-triviality of the holonomy implies that there is a sequence $(y_n)$ in $I$ to converge to $x$ such that each $y_n$ is not a fixed point of $f_v$.
Then each periodic orbit $O(\omega_{f_v}(y_n))$ is a limit circuit, which contradicts to the finiteness of limit cycles.
\end{proof}

We have the following characterization of the finite existence of limit cycles.

\begin{lemma}\label{lem:fin_hol}
The following statements are equivalent for a flow $v$ with finitely sectored singular points on a compact surface $S$:
\\
$(1)$ There are at most finitely many limit cycles.
\\
$(2)$ There are at most finitely many non-trivial circuits with non-trivial holonomy.
\end{lemma}

\begin{proof}
Trivially, the condition $(2)$ implies the condition $(1)$.
Conversely, suppose that there are at most finitely many limit cycles.
Finiteness of sectors of singular points implies that there are at most finitely many non-periodic non-trivial circuits.
By compactness of $S$, there are at most finitely many circuits with non-orientable holonomy.
Since there are at most finitely many limit cycles, Lemma~\ref{lem:dich} implies that there are at most finitely many non-trivial circuits with non-trivial holonomy.
\end{proof}

The inheritance of finitely sectored properties holds as follows.

\begin{lemma}\label{lem:fin_sec}
A flow $v$ with finitely many limit cycles on a compact surface is finitely sectored if and only if so is $v_{\mathrm{col}}$.
\end{lemma}

\begin{proof}
Suppose that each singular point of $v$ is finitely sectored.
Then there are at most finitely many singular points of $v$.
Lemma~\ref{lem:fin_hol} implies that there are at most finitely many non-trivial circuits with non-trivial holonomy.
By construction of $v_{\mathrm{col}}$, we have finitely many new sinks, sources and copies of parts of sectored singular points of $v_{\mathrm{col}}$.
This means that the flow $v_{\mathrm{col}}$ has finitely many sectored singular points.

Conversely, suppose that each singular point of $v_{\mathrm{col}}$ is finitely sectored.
This means that $v$ has at most finitely many non-trivial circuits with non-trivial holonomy and that there are at most finitely many singular points of $v_{\mathrm{col}}$.
Moreover, the number of elliptic (resp. parabolic) sectors of $v$ is less than or equal to one of $v_{\mathrm{col}}$ and so is finite.
From the finite existence of limit cycles, by construction of $v_{\mathrm{col}}$, the difference of the numbers of sectors of $v_{\mathrm{col}}$ and of $v$ are finite.
Since each non-trivial circuits contains finitely many singular points and $v$ has at most finitely many non-trivial circuits with non-trivial holonomy, the flow $v$ have at most finitely many singular points.
Assume that there is a singular point $x$ of $v$ which is not finitely sectored.
Then the image $p_{\mathrm{col}}(p_{\mathrm{me}}^{-1}(x))$ is a finite union of singular points.
By finite sectored property of $v_{\mathrm{col}}$, since only hyperbolic sectors are collapsed via $p_{\mathrm{col}}$, there is a collar of $x$ in $S$ consists of finitely many elliptic and parabolic sectors and infinitely many hyperbolic sectors.
Isolatedness of $x$ implies that the index of $x$ is well-defined and finite.
On the other hand, an insertion of $2k$ hyperbolic sectors add $-2k$ to the index and so the index of $x$ is $- \infty$, which contradicts that the index of $x$ is finite.
Thus each singular point of $v$ is finitely sectored and so the original flow $v$ is finitely sectored.
\end{proof}

A periodic orbit $O$ is isolated if there is its open annular \nbd $\A$ with $\Pv \cap \A = O$.
Moreover, we have the following observations.

\begin{lemma}\label{lem:per_lc}
Any isolated periodic orbit of a flow on a surface is a limit cycle on each side.
\end{lemma}

\begin{proof}
Let $O$ be an isolated periodic orbit with a small open annular \nbd $\A$ with $\Pv \cap \A = O$.
Since $\Sv$ is closed, we may assume that $\A$ contains no singular points.
By the flow box theorem for $O$, the isolated property of the periodic orbit $O$ implies that the cycle $O$ is semi-attracting or semi-repelling on connected components of $\A - O$.
\end{proof}

\begin{corollary}\label{cor:per_lc}
Any periodic orbit of a Morse-Smale-like flow on a surface is a limit cycle.
\end{corollary}

The previous corollary implies that a Morse-Smale-like flow without periodic orbits corresponds to a Morse-like flow.
A separatrix is homoclinic if there is a singular point which is both the $\alpha$-limit set and the $\omega$-limit set.
We have the following dichotomy of non-trivial circuits.

\begin{corollary}\label{cor:dic}
Let $v$ be a quasi-regular Morse-Smale-like flow on a surface $S$ such that each heteroclinic multi-saddle separatrix is contained in the boundary $\partial S$.
Then any non-trivial circuit of $v$ is either a limit circuit or has non-trivial holonomy.
\end{corollary}

\begin{proof}
Corollary~\ref{cor:per_lc} implies that $\Pv$ consists of finitely many limit cycles.
The quasi-regularity implies that each non-trivial circuit is the closure of some homoclinic multi-saddle separatrix.
Fix a non-trivial circuit $\gamma$ with the saddle $x$.
Since fixed points of first return mappings correspond to periodic points, the finite existence of $\Pv$ implies that if the circuit is orientable then the holonomy along the circuit is non-trivial on the left in Figure~\ref{hol_saddle} and so the circuit $\gamma$ is a limit cycle.
\begin{figure}
\begin{center}
\includegraphics[scale=0.35]{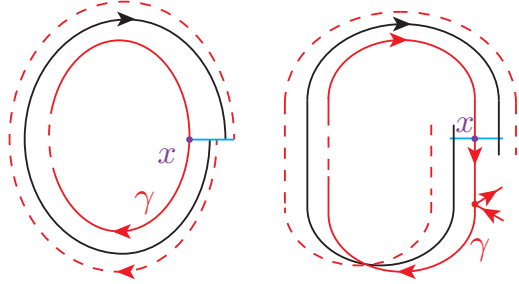}
\end{center}
\caption{Left, an orientable holonomy along a non-trivial circuit which is the closure of some homoclinic saddle separatrix $\gamma$; right, a non-orientable holonomy along a non-trivial circuit which is the closure of some homoclinic saddle separatrix $\gamma$.}
\label{hol_saddle}
\end{figure}%
Thus we may assume that $\gamma$ is non-orientable.
Then the holonomy along $\gamma$ is well-defined as the right in Figure~\ref{hol_saddle}.
Lemma~\ref{lem:dich} implies the assertion.
\end{proof}

We have the following characterization of a quasi-Morse-Smale flow with finitely many singular points.

\begin{proposition}\label{lem:qms_char}
The following are equivalent for a flow $v$ on a compact surface:
\\
{\rm(1)} The flow $v$ is Morse-Smale-like without elliptic sectors.
\\
{\rm(2)} The flow $v$ is quasi-Morse-Smale and has at most finitely many limit cycles and singular points.
\end{proposition}

\begin{proof}
Let $v$ be a flow on a compact surface $S$.
Since a Morse-Smale-like flow has only finitely sectored isolated singular points and so has at most finitely many singular points, the flow $v$ has at most finitely many singular points.

Suppose that $v$ is a Morse-Smale-like flow without elliptic sectors.
Then $v$ has at most finitely many limit cycles.
Let $\Gamma$ be the union of non-trivial circuits with non-trivial holonomy.
Then $S = \mathop{\mathrm{Cl}}(v) \sqcup \mathrm{P}(v)$.
Corollary~\ref{cor:per_lc} implies that any periodic orbits are limit cycles and so that $\mathop{\mathrm{Per}}(v) \subseteq \Gamma$.
By construction, the collapsed flow $v_{\mathrm{col}}$ has no limit circuits and $S_{\mathrm{col}} = \mathop{\mathrm{Cl}}(v_{\mathrm{col}}) \sqcup \mathrm{P}(v_{\mathrm{col}})$.
Since $v$ is a finitely sectored flow without elliptic sectors, Lemma~\ref{lem:fin_sec} implies that so is $v_{\mathrm{col}}$.
Since there are at most finitely many limit circuits of $v$, Lemma~\ref{lem:fin_hol} implies that the union $\Gamma$ of non-trivial circuits with non-trivial holonomy is a finite union of separatrices and closed orbits and so is closed.
Since $S - \Gamma = S_{\mathrm{col}} - \Gamma_{\mathrm{col}}$, this means that $S - \Gamma \subseteq \mathop{\mathrm{Sing}}(v) \sqcup \mathrm{P}(v)$ is open dense in $S$ and so $S_{\mathrm{col}} - \Gamma_{\mathrm{col}} \subseteq  \mathop{\mathrm{Sing}}(v_{\mathrm{col}}) \sqcup \mathrm{P}(v_{\mathrm{col}})$ is open dense in $S_{\mathrm{col}}$.
Since $S$ is a Baire space and  $\mathrm{P}(v)$ is open dense, the intersection $\mathrm{P}(v) \cap (S - \Gamma) = \mathrm{P}(v) \setminus \Gamma$ is open dense.
Since $v$ is finitely sectored, Lemma~\ref{lem:fin_sec} implies that so is $v_{\mathrm{col}}$.
Then $v$ is Morse-Smale-like without elliptic sectors or non-trivial circuits.
Theorem~\ref{main:c} implies that $v_{\mathrm{col}}$ is a gradient flow and so the original flow $v$ is quasi-Morse-Smale.

Conversely, suppose that $v$ is quasi-Morse-Smale and has at most finitely many limit cycles.
Theorem~\ref{main:c} implies that the collapsed gradient flow $v_{\mathrm{col}}$ is a Morse-Smale-like flow without elliptic sectors or non-trivial circuits.
Lemma~\ref{lem:col} implies that $v$ has no elliptic sectors.
Since $v_{\mathrm{col}}$ is gradient, we have $\mathrm{R}(v_{\mathrm{col}}) = \emptyset$ and so $S = \mathop{\mathrm{Cl}}(v) \sqcup \mathrm{P}(v)$.
Since $v_{\mathrm{col}}$ is finitely sectored, Lemma~\ref{lem:fin_sec} implies that so is $v$.
Moreover, we obtain $S_{\mathrm{col}} = \mathop{\mathrm{Sing}}(v_{\mathrm{col}}) \sqcup \mathrm{P}(v_{\mathrm{col}}) = \overline{\mathrm{P}(v_{\mathrm{col}})}$.
The finiteness of singular points of $v_{\mathrm{col}}$ implies that there are at most finitely many limit circuits of $v$ and so that $S = \overline{\mathrm{P}(v)}$.
Thus $v$ is a Morse-Smale-like flow without elliptic sectors.
\end{proof}

The previous proposition implies Theorem~\ref{main:qausi-MS}.

%
%
%
%

\section{Topological characterizations of Morse-Smale flows}

We have the following statements.

\begin{lemma}\label{lem:ch_MSL}
A flow $v$ on a compact surface $S$ is Morse-Smale-like if and only if it is a flow with finitely sectored singular points such that $S = \Cv \sqcup \mathrm{P}(v)$ and that $\Cv$ consists of finitely many orbits.
In any case, any periodic orbits are limit cycles on any sides and the union $\mathrm{P}(v)$ is open dense.
\end{lemma}

\begin{proof}
By definition, a Morse-Smale-like flow is a flow with finitely sectored singular points such that $S = \Cv \sqcup \mathrm{P}(v)$ and that $\Cv$ consists of finitely many orbits.
Conversely, suppose that $v$ is a flow with finitely sectored singular points such that $S = \Cv \sqcup \mathrm{P}(v)$ and that $\Cv$ consists of finitely many orbits.
Lemma~\ref{lem:per_lc} implies any periodic orbits are limit cycles on any sides.
\end{proof}

\begin{lemma}\label{lem:032}
The following are equivalent for a regular Morse-Smale-like flow $v$ on a compact surface $S$:
\\
$(1)$  Each multi-saddle separatrix is contained in the boundary $\partial S$ and there are no homoclinic separatrices.
\\
$(2)$ Each heteroclinic multi-saddle separatrix is contained in the boundary $\partial S$ and any non-trivial circuit with non-trivial holonomy are cycles.
\\
In any case, any non-trivial circuits are limit cycles.
\end{lemma}

\begin{proof}
Let $v$ be a quasi-regular Morse-Smale-like flow on a compact surface $S$.
Lemma~\ref{lem:ch_MSL} implies that $S = \Cv \sqcup \mathrm{P}(v)$ and the union $\mathrm{P}(v)$ is open dense.
\cite{yokoyama2017decompositions}*{Lemma 3.4} implies that each periodic orbit is a limit cycle.
Moreover, each boundary component either is a limit cycle or consists of finitely many singular points and non-recurrent orbits.

Suppose that each multi-saddle separatrix is contained in the boundary $\partial S$ and there are no homoclinic separatrices.
For a non-multi-saddle semi-multi-saddle separatrix, either the $\omega$-limit set is  a semi-attracting circuit or the $\alpha$-limit set is a semi-repelling circuit.
Since $\partial S$ contains no non-periodic circuits as a directed graph, any non-trivial circuits are limit cycles.

Conversely, suppose that each heteroclinic multi-saddle separatrix is contained in the boundary $\partial S$ and that any non-trivial circuit with non-trivial holonomy are cycles.
Assume that there is a homoclinic separatrix $\gamma$ with a saddle $x = \alpha(\gamma) = \omega(\gamma)$.
Then the union $\gamma \sqcup \{ x\}$ is a circle.
If $\gamma \sqcup \{ x\}$ has non-orientable holonomy, then it is a limit cycle, which contradicts that $\gamma \sqcup \{ x\}$ is not a periodic orbit.
Therefore there is an annular collar $\mathbb{A}$ such that $\gamma \sqcup \{ x\}$ is a boundary component of $\mathbb{A}$.
Since $S = \Sv \sqcup \mathrm{P}(v)$, the finiteness of $\Sv$ implies that $\mathbb{A} \subset \mathrm{P}(v)$.
The flow box theorem for $\gamma$ implies that the circuit $\gamma$ is semi-attracting or semi-repelling with respect to $\A$, which contradicts the non-existence of non-periodic limit circuits.
Thus there are no homoclinic separatrices and so each multi-saddle separatrix is contained in the boundary $\partial S$.
\end{proof}

We show that Morse-Smale flows are Morse-Smale-like and state a topological characterization of Morse-Smale flows.

\begin{theorem}\label{MS}
The following are equivalent for a flow $v$ on a compact surface $S$:
\\
$(1)$ The flow $v$ is Morse-Smale.
\\
$(2)$ The flow $v$ is a regular Morse-Smale-like flow such that each limit cycle is topologically hyperbolic and that each multi-saddle separatrix is contained in the boundary $\partial S$.
\\
$(3)$ The flow $v$ is a regular Morse-Smale-like flow such that each non-trivial circuit with non-trivial holonomy is a topologically hyperbolic limit cycle and each heteroclinic multi-saddle separatrix is contained in the boundary $\partial S$.
\\
$(4)$ The flow $v$ is a regular quasi-Morse-Smale flow such that each limit cycle is topologically hyperbolic and that each multi-saddle separatrix is contained in the boundary $\partial S$.
\\
$(5)$ The flow $v$ is a regular quasi-Morse-Smale flow such that each non-trivial circuit with non-trivial holonomy is a topologically hyperbolic limit cycle and that each multi-saddle separatrix is contained in the boundary $\partial S$.
\\
$(6)$ $S = \Cv \sqcup \mathrm{P}(v)$, $\mathrm{P}_{\mathrm{ms}}(v) \subset \partial S$, and the closed point set $\Cv$ consists of finitely many topologically hyperbolic orbits.
\end{theorem}

\begin{proof}
Note that the flow $v$ is regular in any case.
Theorem~\ref{main:qausi-MS} implies that $v$ is quasi-Morse-Smale if and only if $v$ is Morse-Smale-like.
This implies that the conditions $(2)$ and $(4)$ (resp. $(3)$ and $(5)$) are equivalent.
Lemma~\ref{lem:ch_MSL} implies that the conditions $(2)$ and $(6)$ are equivalent.
Lemma~\ref{lem:032} implies that that conditions $(2)$ and $(3)$ are equivalent.

Suppose that $v$ is Morse-Smale.
Then each limit circuit is a topologically hyperbolic limit cycle and each multi-saddle separatrix is contained in the boundary $\partial S$.
Therefore there are neither non-periodic limit circuits, topologically non-hyperbolic limit cycles, homoclinic separatrices, nor non-limit circuits with non-orientable holonomy.
The resulting flow $w$ by removing limit circuits, taking metric completion, and collapsing new boundary components into singletons is a quasi-regular gradient flow.
This means that $v$ is quasi-Morse-Smale without non-periodic limit circuits, topologically non-hyperbolic limit cycles, homoclinic separatrices, or non-limit circuits with non-orientable holonomy.
Hence $v$ is a quasi-Morse-Smale flow such that any limit circuits are topologically hyperbolic cycle and that any multi-saddle separatrices are heteroclinic contained in the boundary $\partial S$.
Therefore the condition $(1)$ implies the conditions $(4)$ and $(5)$.

Suppose that $v$ is a regular Morse-Smale-like flow such that each limit cycle is topologically hyperbolic and that each multi-saddle separatrix is contained in the boundary $\partial S$.
Then any non-trivial circuits are limit cycles.
By Theorem~\ref{main:qausi-MS}, the resulting flow $v_{\mathrm{col}}$ from $v$ by removing limit cycles, taking metric completion, and collapsing new boundary components into singletons is a regular gradient flow without non-trivial circuits such that each multi-saddle separatrix of $v_{\mathrm{col}}$ is contained in the boundary $\partial S_{\mathrm{col}}$.
Theorem~\ref{main:c} implies that $v_{\mathrm{col}}$ is regular Morse-Smale-like without non-trivial circuits such that each multi-saddle separatrix of $v_{\mathrm{col}}$ is contained in the boundary $\partial S_{\mathrm{col}}$.
Theorem~\ref{main:Morse} implies that $v_{\mathrm{col}}$ is Morse.
By construction of $v_{\mathrm{col}}$, the original flow $v$ is Morse-Smale.
\end{proof}

Theorem~\ref{main:d} is followed from the previous theorem.

\section{Genericity of Morse flows and Morse-Smale flows}

To characterize Morse (resp. Morse-Smale) flows and generic non-Morse gradient (resp. non-Morse-Smale Morse-Smale-like) flows, we describe the stability of transversality, quasi-regularity, and non-existence of multi-saddle separatrices outside of the boundary.
%

\subsection{Classes of flows on surfaces}

Let $S$ be a compact surface.
Denote by $\mathcal{G}(S)$ (resp. $\mathcal{Q}(S)$) the set of gradient (resp. Morse-Smale-like) flows with finitely many singular points on a compact surface $S$, by $\mathcal{G}^r(S)$ (resp. $\mathcal{Q}^r(S)$) the subset of $C^r$ flows in $\mathcal{G}(S)$ (resp. $\mathcal{Q}(S)$) with the $C^r$ topology for any $r \in \mathbb{Z}_{\geq 0}$, and by $\mathcal{G}_*(S)$ (resp. $\mathcal{Q}_*(S)$) the subset of quasi-regular gradient flows without fake saddles on $S$.

For any $r \in \mathbb{Z}_{\geq 0}$, let $\mathcal{G}^r_*(S) := \mathcal{G}^r(S) \cap \mathcal{G}_*(S)$ be the set of $C^r$ gradient flows without fake saddles on a compact surface $S$ and $\mathcal{Q}^r_*(S) := \mathcal{Q}^r(S) \mathcal{Q}_*(S)$ the set of $C^r$ Morse-Smale-like flows with finitely many singular points and  without fake saddles.
Note that a $C^{r+1}$ flow generates a $C^r$ vector field but that a $C^r$ vector field need not generate a $C^{r+1}$ flow in general.
%

For any $k \in \mathbb{Z}_{>0}$, denote by $\mathcal{G}_{k/2}(S)$ (resp. $\mathcal{Q}_{k/2}(S)$) the set of gradient (resp. quasi-Morse-Smale) flows on a compact surface $S$ whose sum of indices of sources, sinks, $\partial$-sources is $k/2$.
For any $l \in \mathbb{Z}_{>0}$, denote by $\mathcal{Q}_{k/2,l}(S)$ the subset of Morse-Smale-like flows in $\mathcal{Q}_{k/2}(S)$ each of whose number of limit circuits is $l$.
Put $\mathcal{Q}^r_{k/2}(S) := \mathcal{Q}_{k/2}(S) \cap \mathcal{Q}^r(S)$, $\mathcal{Q}^r_{k/2,l}(S) := \mathcal{Q}_{k/2,l}(S) \cap \mathcal{Q}^r(S)$, $\mathcal{G}_{k/2,*}(S) := \mathcal{G}_{k/2}(S) \cap \mathcal{G}_*(S)$, and $\mathcal{Q}_{k/2,l,*}(S) := \mathcal{Q}_{k/2}(S) \cap \mathcal{Q}_*(S)$, $\mathcal{G}^r_{k/2,*}(S) := \mathcal{G}^r_{k/2}(S) \cap \mathcal{G}_*(S)$, and $\mathcal{Q}^r_{k/2,l,*}(S) := \mathcal{Q}^r_{k/2}(S) \cap \mathcal{Q}_*(S)$.


\subsection{Stability of transversality of $C^1$ flows}

By the Thom transversality theorem, we have the following statement.

\begin{lemma}\label{transversality}
For any $s \in \mathbb{Z}_{>0}$ and for a flow $v \in \mathcal{G}^s(S)$ $(\mathrm {resp.} \, v \in \mathcal{Q}^s(S))$ on a compact surface $S$, any closed transversal and any transverse closed arc are invariant under $C^s$ perturbation in $\mathcal{G}^s(S)$ $(\mathrm {resp.} \,  \mathcal{Q}^s(S))$.
\end{lemma}

In this surface case, the proof is straightforward, and so we state as follows.

\begin{proof}
Denote $\chi^s$ by either $\mathcal{G}^s(S)$ or $v \in \mathcal{Q}^s(S)$.
Fix a flow $v \in \chi^s$.
Let $T$ be a closed transversal or a transverse closed arc to $v$.
Fix a Riemannian metric on $S$.
Since $S$ is two dimensional, the orthonormal subspace $(T_p T)^{\perp}$ of $T_p S$ at any point $p \in T$ is defined and one-dimensional.
Let $\pi: T_p S \to T_p T$ be the canonical projection.
Compactness of $T$ implies that the minimal value $\min_{p \in T} \left\| \pi\left(\frac{\partial}{\partial t}v(p)\right) \right\|$ is positive.
Take a $C^s$ \nbd $\mathcal{U}$ of $v$ in $\chi^s$ such that $\min_{x \in S} \left\| \frac{\partial}{\partial t}v(x) - \frac{\partial}{\partial t}w(x) \right\| < \min_{p \in T} \left\| \pi\left(\frac{\partial}{\partial t}v(p)\right) \right\|$ for any flow $w \in \mathcal{U}$.
For any flow $w \in \mathcal{U}$, we have that $\min_{p \in T} \left\| \frac{\partial}{\partial t}v(p) - \frac{\partial}{\partial t}w(p) \right\| < \min_{p \in T} \left\| \pi\left(\frac{\partial}{\partial t}v(p)\right) \right\|$ and so that the minimal value $\min_{p \in T} \left\| \pi\left(\frac{\partial}{\partial t}w(p)\right) \right\|$ is positive.
Therefore $T$ is transverse to any flow in $\mathcal{U}$.
\end{proof}

\subsection{Stability of quasi-regularity}

Put $\chi := \mathcal{G}$ or $\mathcal{Q}$.
Equip  $\chi^r_{k/2}(S)$ (resp. $\chi^r_*(S)$, $\chi^r_{k/2}(S)$, and $\chi^r_{k/2,l}(S)$) with the $C^r$ topologies.
Then $\chi_*(S) = \mathcal{G}_*(S)$ or $\mathcal{Q}_*(S)$, $\chi_{k/2}(S) = \mathcal{G}_{k/2}(S)$ or $\mathcal{Q}_{k/2}(S)$, and so on.
%

For any flow $v \in \chi_*(S)$ and an open subset $U$ of $S$ with $\mathop{\mathrm{Sing}}(v) \cap \partial U = \emptyset$, denote by $\mathrm{ind}_v(U)$ the sum of indices of singular points of $v$ in $U$.
To show the open denseness of regularity, we show following statements.

\begin{lemma}\label{index}
Let $v$ be a flow in $\chi^s_*(S)$ for any $s \in \mathbb{Z}_{>0}$, $x \in S \setminus \partial S$ an isolated singular point, and $U$ an open simply connected \nbd of $x$ with $\{ x \} = \overline{U} \cap \mathop{\mathrm{Sing}}(v)$ whose boundary consists of finitely many transverse closed arcs and non-degenerate closed orbit arcs
Then there is a \nbd $\mathcal{U} \subset \chi^s_*(S)$ of $v$ with $\mathrm{ind}_v(U) =  \mathrm{ind}_w(U)$ for any flow $w \in \mathcal{U}$.
\end{lemma}

\begin{proof}
Recall that the index of an isolated singular point outside of the boundary of $S$ is determined by the numbers of inner and outer tangencies of a loop which is transverse at all but finitely many points and bounds an open disk containing the singular point by Poincar\'e-Hopf theorem for continuous flows with finitely many singular points on compact surfaces.
Therefore the index of such a singular point is determined by the numbers of orbit arcs near inner and outer tangencies of a loop which consists of finitely many non-degenerate orbit arcs and transverse arcs and which bounds an open disk containing the singular point.

We claim that we can take a loop that is near the original loop with respect to the Hausdorff distance such that they have the same cyclic arrangements of alternative sequences of transverse arcs and directed orbit arcs with positive/negative flow directions
under any small perturbation.
Indeed, let $\gamma$ be a loop which consists of finitely many non-degenerate orbit arcs $C_1, \ldots , C_k$ near tangencies of $\gamma$ respectively and a union $T'$ of transverse closed arcs to $v$ such that $\partial C_i = C_i \cap T'$ and $T' \subseteq \gamma$
as in Figure~\ref{trans_arcs}.
\begin{figure}
\begin{center}
\includegraphics[scale=0.5]{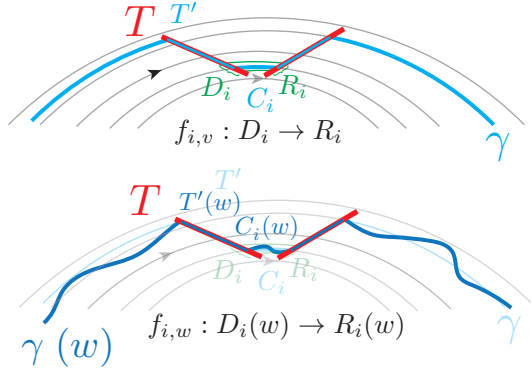}
\end{center}
\caption{A loop which consists of finitely many non-degenerate orbit arcs and transverse arcs and its perturbed loop which consists of finitely many non-degenerate orbit arcs}
\label{trans_arcs}
\end{figure}
Take a union $T$ of transverse closed arcs with $T' \subset \mathrm{int} T$.
Then we can define a first return map $f_{i,v}: D_i \to R_i$ along $C_i$ in $\gamma$ for
some transverse open arcs $D_i$ and $R_i$ contained $T$ with $\partial C_i = C_i \cap (D_i \sqcup R_i)$.
By Lemma~\ref{transversality}, there is a \nbd $\mathcal{U} \subset \chi_*(S)$ of $v$ such that for any flow $w \in \mathcal{U}$ such that $T$ is also a union of transverse closed arcs and such that there are transverse open arcs $D_i(w) \subseteq D_i$ and $R_i(w) \subset R_i$, and the first return map $f_{i,w}: D_i(w) \to R_i(w)$ along $w$ whose orientation is same as $f_{i,v}: D_i \to R_i$.
Fix any flow $w \in \mathcal{U}$.
There are a union $T'(w) \subset T$ of transverse closed arcs of $w$ and non-degenerate orbit arcs $C_1(w), \ldots , C_k(w)$ with $\partial C_i(w) = C_i(w) \cap (D_i(w) \cup R_i(w))$ near $C_1, \ldots , C_k$ respectively such that the union $\gamma(w) := T'(w) \cup \bigcup_{i = 1}^k C_i(w)$ is a loop.
Since the orientation of the first return map $f_{i,w}: D_i(w) \to R_i(w)$ is same as one of $f_{i,v}: D_i \to R_i$, this implies that $\gamma$ and $\gamma(w)$ have the same cyclic arrangements of alternative sequences of transverse arcs and directed orbit arcs with positive/negative flow directions.

Replacing orbit arcs into pairs of transverse arcs as in Figure~\ref{modified}, we obtain loop $\mu$ from $\gamma$ (resp. $\mu(w)$ from $\gamma(w)$) which  is transverse except finitely many tangencies.
\begin{figure}
\begin{center}
\includegraphics[scale=0.4]{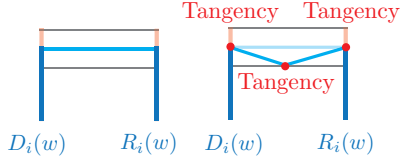}
\end{center}
\caption{Replacement of an orbit arc into a pair of two transverse arcs.}
\label{modified}
\end{figure}
Then any triples of new tangencies have exactly either one or two outer tangencies depending on the positions of orbit arcs $C_i$ and $C_i(w)$ in relation to the bounded disks.
Therefore the number of inner (resp. outer) tangencies of the loop $\mu$ for $v$ equals one of the loop $\mu(w)$ for $w$.
\end{proof}

\begin{lemma}\label{index_bdry}
Let $v$ be a flow in $\chi^s_*(S)$ for any $s \in \mathbb{Z}_{>0}$, $x \in \partial S$ an isolated singular point, and $U$ an open simply connected \nbd of $x$ with $\{ x \} = \overline{U} \cap \mathop{\mathrm{Sing}}(v)$ such that $\overline{\partial U \setminus \partial S}$ consists of finitely many transverse closed arcs and non-degenerate closed orbit arcs.
Then there is a \nbd $\mathcal{U} \subset \chi^s_*(S)$ of $v$ with $\mathrm{ind}_v(U) =  \mathrm{ind}_w(U)$ for any flow $w \in \mathcal{U}$.
\end{lemma}

\begin{proof}
Suppose that $x \in \partial S$.
Taking the double $S_{\mathrm{dbl}} := S \sqcup -S$ of $S$, we can obtain the induced flow $w_{\mathrm{dbl}}$ on $S_{\mathrm{dbl}}$ from any flow $w \in \chi^s_*(S)$.
For any $C^1$-\nbd $\mathcal{U}_{v_{\mathrm{dbl}}}$ of $v_{\mathrm{dbl}}$, there is a $C^1$-\nbd $\mathcal{U}_{v}$ of $v$ such that $\{ w_{\mathrm{dbl}} \mid w \in \mathcal{U}_v \} \subset \mathcal{U}_{v_{\mathrm{dbl}}}$.
Since the double $S_{\mathrm{dbl}}$ is a closed surface and so $x \in S_{\mathrm{dbl}} = S_{\mathrm{dbl}} - \partial S_{\mathrm{dbl}}$, Lemma~\ref{index} implies that the assertion holds.
\end{proof}

Non-degenerate singular points have the same topological type if each of them is either a center, a saddle, a $\partial$-saddle, a sink, a $\partial$-sink, a source, or a $\partial$-source.
The previous lemmas imply the following properties.


\begin{lemma}\label{m2}
The regularity condition for singular points is dense in $\chi^r_{k/2}(S)$ for any $r \in \mathbb{Z}_{\geq 0}$ and open dense in $\chi^s_{k/2}(S)$ for any $s \in \mathbb{Z}_{> 0}$.
Moreover, the topological types of any singular points in the open dense subset in $\chi^s_{k/2}(S)$ are invariant under small perturbations.
\end{lemma}

\begin{proof}
Since any multi-saddle can be approximated by saddles and $\partial$-saddles, the regularity condition for singular points is a dense condition in $\chi^r_{k/2}(S)$ for any $r \in \mathbb{Z}_{\geq 0}$.
Recall that the index of a sink/source (resp. $\partial$-sink/$\partial$-source) is $1$ (resp. $1/2$) and that the index of a $k$-saddle (resp. $k/2$-$\partial$-saddle ) for any $k \in \mathbb{Z}_{>0}$ is $-k$ (resp. $-k/2$).
Note that two singular points outside of $\partial S$ which are of different types have different indices.
On the other hand, the indices of a sink/source, a $\partial$-sink/$\partial$-source, a saddle, a $\partial$-saddle are different from each other because of the indices $1, 1/2, -1, -1/2$.
Notice that each singular point of a flow in $\chi^s_{k/2}(S)$ is either a sink, a source, a $\partial$-sink, a $\partial$-source, a saddle, or $\partial$-saddle.
By definition of $\chi_{k/2}(S)$, the sum of indices of parabolic sectored singular points of a flow in $\chi^s_{k/2}(S)$ is $k/2$.

Let $v$ be a flow in $\chi^s_{k/2}(S)$.
Then the sum of positive indices of singular points of $v$ is $k/2$.
Fix a disjoint union $U$ of simply connected open \nbds of the singular points of $v$ such that any connected component of $U$ contains exactly one singular point of $v$  and that for any connected component $V$ the closure $\overline{\partial U \setminus \partial S}$ consists of finitely many transverse closed arcs and non-degenerate closed orbit arcs.
Fix any flow $w$ in $\chi^s_{k/2}(S)$ near $v$.
Lemma~\ref{transversality} implies that $U$ is also a disjoint union of simply connected open\nbds of the singular points of $w$.
Poincar\'e-Hopf theorem implies that $\mathrm{ind}_{v}(S) = \mathrm{ind}_{w}(S)$.
Lemma~\ref{index} and Lemma~\ref{index_bdry} imply that $\mathrm{ind}_{w,+}(S) \geq k/2$.
Since the indices of a sink/source, a $\partial$-sink/$\partial$-source, a saddle, a $\partial$-saddle are different from each other, if there are a singular point of $v$ in a connected component of $U$ and a singular point of $w$ in the component which are of different types then $\mathrm{ind}_{w,+}(S) > k/2$, which contradicts $\mathrm{ind}_{w,+}(S) = k/2$.
Thus the regularity condition for singular points is open and the topological types of any singular points in the open dense subset in $\chi^s_{k/2}(S)$ are invariant under small perturbations.
\end{proof}

It is known that the non-degeneracy of $C^1$ Morse-Smale vector fields (and so $C^2$ Morse-Smale flows) is an open condition.
On the other hand, the non-degeneracy of $C^1$ Morse-Smale flows is not an open condition in general.

\subsection{Stability of non-existence of multi-saddle separatrices outside of the boundary}

Recall that the condition $(\mathrm{M}4)$ is the condition that any multi-saddle separatrices are contained in the boundary $\partial S$.
We show the stability of non-existence of multi-saddle separatrices outside of the boundary.


\begin{lemma}\label{m4}
For any $r \in \mathbb{Z}_{\geq 0}$, the subset of regular flows in $\chi^r_{k/2}(S)$ whose multi-saddle separatrices are contained in the boundary $\partial S$ is dense in $\chi^r_{k/2}(S)$.
\end{lemma}

\begin{proof}
Theorem~\ref{main:a} implies that any flow in $\chi^r_{k/2}(S)$ is finitely sectored.
The generalization of the Poincar\'e-Bendixson theorem for a flow with finitely many singular points implies that each of the $\omega$-limit set and the $\alpha$-limit set of a point of a flow in $\chi^r_{k/2}(S)$ is a singular point or a limit circuit.

First, we show the denseness in $\chi^r_{k/2}(S)$.
Let $v$ be a flow in $\chi^r_{k/2}(S)$.
By Lemma~\ref{m2}, we may assume that $v$ is regular.
Since any flow in $\chi_{k/2}(S)$ is of weakly finite type, \cite{yokoyama2017decompositions}*{Corollary 7.8} implies that each connected component of the complement of $\mathop{\mathrm{BD}}_+(v)$ is either an open transverse annulus or an open invariant flow box such that the union $\mathop{\mathrm{BD}}_+(v)$ is the finite union of singular points, limit cycles, and semi-multi-saddle separatrices of $v$.
Moreover, $\omega$-limit (resp. the $\alpha$-limit) sets of any points in a connected component of $\mathop{\mathrm{BD}}_+(v)$ correspond to each other.
Therefore we can perturb any multi-saddle separatrix outside of the boundary $\partial S$
as in Figure~ \ref{perturbation} using an open transverse annulus or an open invariant flow box which is a connected component of $S - \mathop{\mathrm{BD}}_+(v)$.
\begin{figure}
\begin{center}
\includegraphics[scale=0.4]{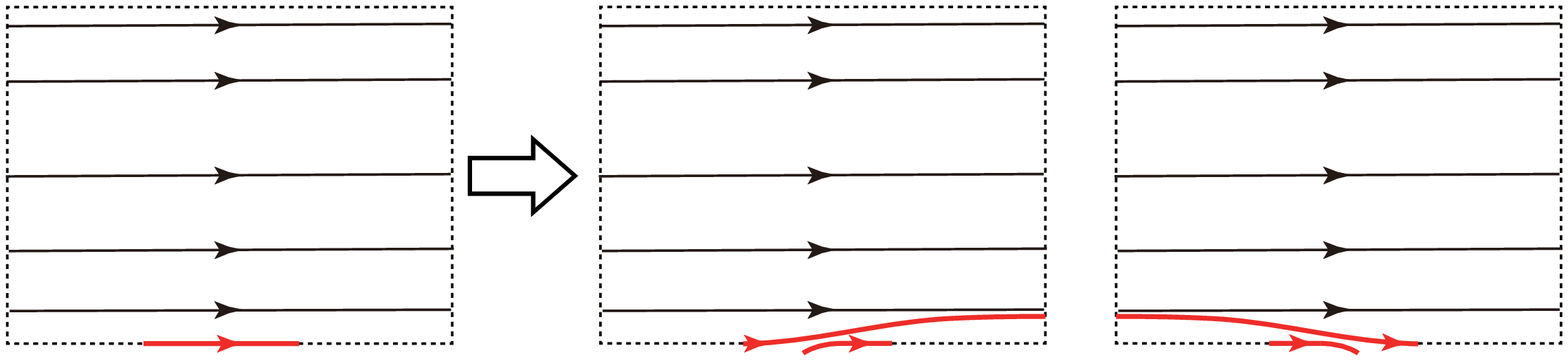}
\end{center}
\caption{Annihilation of a multi-saddle separatrix and creation of a topological limit cycle}
\label{perturbation}
\end{figure}
Then this perturbation preserves a \nbd of the singular point set and $\omega$-limit and the $\alpha$-limit sets.
This implies that the property of finitely sectored singular points is invariant under the perturbations.

We claim that we can choose the perturbation in $\chi^r_{k/2}(S)$ which annihilate a multi-saddle separatrix outside of $\partial S$ in $\chi^r_{k/2}(S)$ respectively.
Indeed, let $\gamma$ be a multi-saddle separatrix outside of $\partial S$ and $x \in \gamma$.
Let $w$ be such the resulting flow of $v$ by the perturbation by applying to a multi-saddle separatrix outside of $\partial S$ the operation as in Figure~\ref{perturbation}.
Suppose that $x$ is contained in the boundary of an invariant flow box $B$.
Then the $\omega$-limit set $\omega_w(x)$ is a singular point or a limit circuit.
Moreover, the perturbation preserves limit circuits.
By construction, the non-existence of non-closed recurrent points and the open denseness of non-recurrent points.
This means that the perturbation can be realized in $\chi^r_{k/2}(S)$ and is desired.
Thus we may assume that $x$ is contained in the boundary of a semi-repelling transverse annulus $A$.
Then $x$ is contained in a limit circuit $\gamma$ for $v$.
Let $A$ be a semi-repelling/semi-attracting basin of the limit circuit $\gamma$.
We can choose the perturbation as in Figure~\ref{perturbation02} which annihilate a multi-saddle separatrix in $\gamma$, creates a topological limit cycle, and preserves the non-existence of non-closed recurrent points and the open denseness of non-recurrent points.
\begin{figure}
\begin{center}
\includegraphics[scale=0.5]{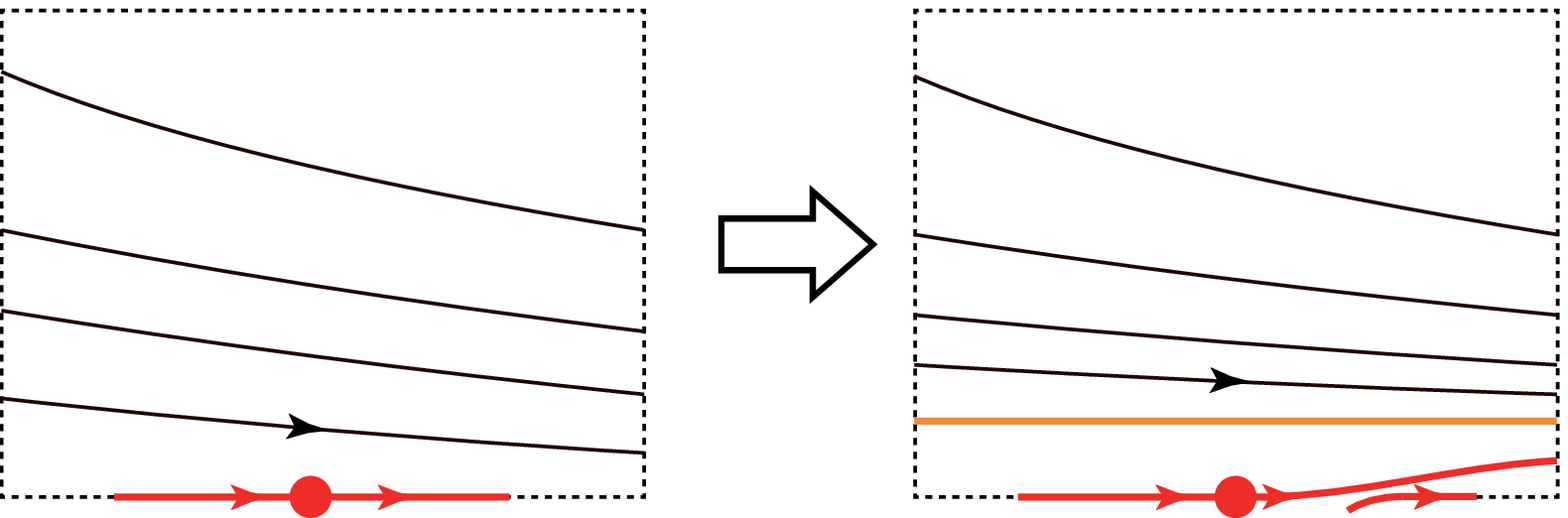}
\end{center}
\caption{Annihilation of a multi-saddle separatrix and creation of a topological limit cycle}
\label{perturbation02}
\end{figure}
This implies that the perturbation can be realized in $\chi^r_{k/2}(S)$ and is desired.

Applying the claim finitely many times, the multi-saddle separatrices become contained in the boundary $\partial S$.

\end{proof}

Notice that the proof of Lemma~\ref{m4} says that any small perturbations do not increase the number of multi-saddle separatrices outside of $\partial S$ under the regularity condition and that the proof of Lemma~\ref{m2} says that topological hyperbolic singular points are invariant under any small perturbations.
Moreover, the existence of fake limit cycles breaks the openness of the condition $(\mathrm{M}4)$, and that limit cycles can be bifurcated into topologically non-hyperbolic limit cycles.
However, restricting the subspace of flows, forbidding the existence of fake limit cycles, and fixing the number of limit cycles, we have the following openness.

\begin{lemma}\label{m45}
For any $s \in \mathbb{Z}_{>0}$, under the non-existence of fake limit cycles, the subset of regular flows in $\chi^s_{k/2,*}(S)$ whose multi-saddle separatrices are contained in the boundary $\partial S$ is open dense in $\chi^s_{k/2,*}(S)$.
\end{lemma}

\begin{proof}
Lemma~\ref{m2} implies that the regularity condition for singular points is open dense and the topological types of any singular points are invariant under small perturbations.
The density of the condition that multi-saddle separatrices are contained in the boundary $\partial S$ is followed from Lemma~\ref{m4}.
Let $v$ be a regular flow in $\chi_{k/2,*}(S)$ whose multi-saddle separatrices are contained in the boundary $\partial S$.
Since the non-existence of multi-saddle separatrices in the interior $\mathrm{int} S$ is an open condition,  the multi-saddles are contained in the boundary $\partial S$ under any small perturbations.
\end{proof}

Theorem~\ref{main:d} and Lemma~\ref{m45} imply Theorem~\ref{th:open_dense_M} and Theorem~\ref{th:open_dense_MS}.
The authors would like to know whether the set of Morse-Smale flows on a compact surface $S$ in $\mathcal{G}^0_{k/2,*}(S)$ {\rm(resp.} $\mathcal{Q}^0_{k/2}(S)${\rm)} is open dense in $\mathcal{G}^0_{k/2,*}(S)$ {\rm(resp.} $\mathcal{Q}^0_{k/2}(S)${\rm)}.

\subsection{Higher codimensional transition spaces of Morse-Smale flows}

For any $m \in \mathbb{Z}_{>0}$, define the following conditions $(\mathrm{M}4_m)$:
\\
$(\mathrm{M}4_m)$ There are exactly $m$ multi-saddle separatrices outside of the boundary $\partial S$.
\\
Note that the condition $(\mathrm{M}4)$ equals the condition $(\mathrm{M}4_0)$.
For any $k, m \in \mathbb{Z}_{>0}$, denote by $\mathcal{G}_{k/2,m,*}(S)$ the subset of quasi-regular gradient flows whose sum of indices of parabolic sectored singular points is $k/2$ and which satisfy the condition $(\mathrm{M}4_{m})$.
For any $k, l, m \in \mathbb{Z}_{>0}$, denote by $\mathcal{Q}_{k/2,l,m,*}(S)$ the subset of quasi-regular Morse-Smale-like flows in $\mathcal{Q}_{k/2}(S)$ each of whose number of limit circuits is $l$ and which satisfy the condition $(\mathrm{M}4_{m})$.
The proof of Lemma~\ref{m4} implies an analogous statement.

\begin{lemma}\label{m4k}
Denote $(\chi^r_m)$ be a decreasing sequence of $\mathcal{G}^r_{k/2,m,*}(S)$  for any $k, m \in \mathbb{Z}_{>0}$ {\rm(resp.} $\mathcal{Q}^r_{k/2,l,m,*}(S)$ for any $k, l, m \in \mathbb{Z}_{>0} )$.
The complement $\bigsqcup_{i\geq m}\chi^r_i - \chi^r_{m+1}$ is dense in $\bigsqcup_{i\geq m}\chi^r_i$.
Under the non-existence of fake limit cycles, the complement $\bigsqcup_{i\geq m}\chi^r_i - \chi^r_{m+1}$ is open in $\bigsqcup_{i\geq m}\chi^r_i$.
\end{lemma}

\section{Generic non-Morse gradient flows}

In this section, we characterize and list all generic intermediate Morse-like flows between gradient flows.

\subsection{Generic transition rules of Morse flows}

Splittings of multi-saddles as in Figure~\ref{split_multi_saddle} imply the following observations.
\begin{figure}
\begin{center}
{\includegraphics[height=0.2\textheight]{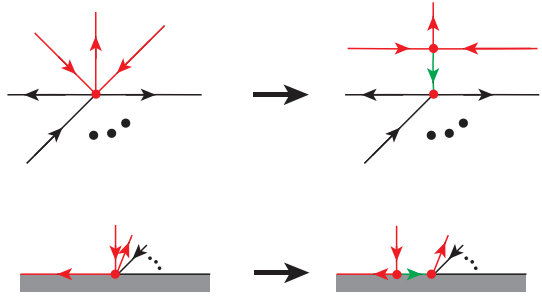}}
\end{center}
\caption{Splittings of multi-saddles creating multi-saddle separatrices.}
\label{split_multi_saddle}
\end{figure}

\begin{lemma}\label{lem:multi-saddle}
Let $v \in \mathcal{G}_{k/2,*}^r(S)$ be a flow for any $r \in \mathbb{Z}_{\geq 0}$.
Any $(1+k)$-saddle {\rm(resp.} $(1+k)/2$-$\partial$-saddle{\rm)} for any positive integer $k$ can be split into a pair of a saddle and a $k$-saddle {\rm(resp.} $\partial$-saddle and $k/2$-$\partial$-saddle{\rm)} with a multi-saddle separatrix between them by an arbitrarily small $C^r$-perturbation.
\end{lemma}

Recall $\mathcal{M}_{k/2}^r(S)^c = \mathcal{G}_{k/2,*}^r(S) - \mathcal{M}_{k/2}^r(S)$.
Theorem~\ref{main:Morse} implies the following observation.

\begin{lemma}\label{lem:Morse_dic}
A flow $v \in \mathcal{G}_{k/2,*}^r(S)$ for any $r \in \mathbb{Z}_{>0}$ is not Morse \rm{(i.e.} $v \in \mathcal{M}_{k/2}^r(S)^c${\rm)} if and only if  one of the following statements holds:
\\
$(1)$ $\mathrm{P}_{\mathrm{ms}}(v) \not\subset \partial S$ {\rm(i.e.} there are multi-saddle separatrices outside of $\partial S$ {\rm)}.
\\
$(2)$ There are multi-saddles that are not topologically hyperbolic.
\end{lemma}


\subsection{Proof of Theorem~\ref{main:e}}

Fix a flow $v \in \mathcal{M}_{k/2}^r(S)^c$.
Lemma~\ref{lem:Morse_dic} implies that either $\mathrm{P}_{\mathrm{ms}}(v) \not\subset \partial S$ or $v$ has multi-saddles which are not topologically hyperbolic.

We claim that the density holds.
Indeed, suppose that $\mathrm{P}_{\mathrm{ms}}(v) \not\subset \partial S$ {\rm(i.e.} there are multi-saddle separatrices outside of $\partial S$ {\rm)}.
Then we can split such multi-saddles into saddles and $\partial$-saddles with multi-saddle separatrices by arbitrarily small $C^r$-perturbations as in Figure~\ref{split_multi_saddle} because of Lemma~\ref{lem:multi-saddle}.
Applying such splittings finitely many times, we obtain a regular flow in $\mathcal{M}_{k/2}^r(S)^c$ such that there are multi-saddle separatrices outside of $\partial S$.
Lemma~\ref{m4} implies that we can obtain an $h$-unstable flow in $\mathcal{G}_{k/2,*}(S)$.
Thus we may assume that $v$ has topological non-hyperbolic multi-saddles.
By Lemma~\ref{m4}, by perturbing multi-saddle separatrices, we may assume that the multi-saddle separatrices are contained in $\partial S$.
From Lemma~\ref{lem:multi-saddle}, we can split any multi-saddles on (resp. outside of) $\partial S$ into $\partial$-saddles (resp. saddle with multi-saddle separatrices) by arbitrarily small $C^r$-perturbations.
Thus we may assume that any topological non-hyperbolic multi-saddles are contained in $\partial S$.
Splitting all multi-saddles except one, we obtain a resulting flow in $\mathcal{G}_{k/2,*}(S)$ which is $p$-unstable.

We claim that the openness holds.
Indeed, suppose that $v$ is $p$-unstable.
The condition of fixing the sum of indices of parabolic sectored singular points implies that each topological hyperbolic singular point is preserved by any small perturbations in $\mathcal{G}_{k/2,*}^r(S)$.
Since the non-existence of multi-saddle separatrices outside of $\partial S$ are open, the $1$-$\partial$-saddle are preserved by any small perturbations in $\mathcal{M}_{k/2}^r(S)^c$.
This means that the $p$-unstable flows in $\mathcal{M}_{k/2}^r(S)^c$ forms an open subset.
Suppose that $v$ is $h$-unstable.
As above, each topological hyperbolic singular point is preserved by any small perturbations in $\mathcal{G}_{k/2,*}^r(S)$.
From Lemma~\ref{lem:Morse_dic}, the unique existence of a multi-saddle separatrix outside of $\partial S$ is open in $\mathcal{M}_{k/2}^r(S)^c$.
This means that the $h$-unstable flows in $\mathcal{M}_{k/2}^r(S)^c$ forms an open subset.

\subsection{Proof of Theorem~\ref{th:generic_gradient_trans}}
From Theorem~\ref{th:open_dense_M}, the subset of Morse flows in $\mathcal{G}^r_{k/2,*}(S)$ are open dense in $\mathcal{G}^r_{k/2,*}(S)$.
Since any gradient flows has no homoclinic separatrices, any multi-saddle separatrices are  heteroclinic.
Theorem~\ref{main:e} says that any generic gradient flows in $\mathcal{G}_{k/2,*}(S)$ are $h$-unstable or $p$-unstable.
Since each heteroclinic multi-saddle separatrix connects between saddles and $\partial$-saddles, all generic transitions among gradient flows in $\mathcal{G}^r_{k/2,*}(S)$ are listed as in Figure~\ref{trans_all}.

\section{Generic non-Morse-Smale ``gradient flows with limit cycles''}

In this section, we characterize and list all generic intermediate Morse-Smale-like flows between gradient flows.

\subsection{Generic transitions of Morse-Smale flows}

Theorem~\ref{MS} implies the following observations.

\begin{lemma}\label{lem:Morse_dic_ms}
For any $r \in \mathbb{Z}_{\geq 0}$, a flow $v \in \mathcal{Q}_{k/2,*}^r(S)$ is not Morse-Smale if and only if one of the following statements holds:
\\
$(1)$ $\mathrm{P}_{\mathrm{ms}}(v) \not\subset \partial S$ {\rm(i.e.} there are multi-saddle separatrices outside of $\partial S$ {\rm)}.
\\
$(2)$ There are multi-saddles that are not topologically hyperbolic.
\\
$(3)$ There are fake limit cycles.
\end{lemma}

\begin{corollary}\label{cor:Morse_dic_ms}
For any $r \in \mathbb{Z}_{\geq 0}$, a flow $v \in \mathcal{Q}_{k/2,**}^r(S)$ is not Morse-Smale \rm{(i.e.} $v \in \mathcal{MS}_{k/2,l}^r(S)^c${\rm)} if and only if  one of the following statements holds:
\\
$(1)$ $\mathrm{P}_{\mathrm{ms}}(v) \not\subset \partial S$.
\\
$(2)$ There are multi-saddles that are not topologically hyperbolic.
\end{corollary}

We generalize $p$-unstability and $h$-unstability for Morse-Smale-like flows as follows:
A flow in $\mathcal{Q}_{k/2,l,**}(S)$ which satisfies the conditions $(\mathrm{M}3)$ and $(\mathrm{M}4)$ is $p$-unstable if all singular points except one $1$-$\partial$-saddle are topologically hyperbolic.
In other words, a flow $v$ is $p$-unstable in $\mathcal{Q}_{k/2,l,**}(S)$ if and only if $v$ is a quasi-regular Morse-Smale-like flow $v$ with $\mathrm{P}_{\mathrm{ms}}(v) \subset \partial S$ without topologically non-hyperbolic limit cycles such that all singular points except one $1$-$\partial$-saddle are topologically hyperbolic.
A flow in $\mathcal{Q}_{k/2,**}(S)$ which satisfies the conditions $(\mathrm{m}2_{+++})$ and $(\mathrm{M}3)$ is $h$-unstable if there are exactly one multi-saddle separatrix outside of the boundary $\partial S$.
In other words, a flow $v$ is $h$-unstable in $\mathcal{Q}_{k/2,l,**}(S)$ if and only if $v$ is a regular Morse-Smale-like flow $v$ without topologically non-hyperbolic limit cycles such that there is exactly one multi-saddle separatrix outside of the boundary $\partial S$.
By definition, any $p$-unstable flow and any $h$-unstable flow in $\mathcal{Q}_{k/2,**}(S)$ are contained in $\mathcal{MS}_{k/2,l}^r(S)^c$.
%
As the gradient flow case, we have the following statements

\begin{lemma}\label{lem:multi-saddle_ms}
Let $v \in \mathcal{Q}_{k/2,**}^r(S)$ be a flow for any $r \in \mathbb{Z}_{\geq 0}$.
Any $(1+k)$-saddle {\rm(resp.} $(1+k)/2$-$\partial$-saddle{\rm)} for any positive integer $k$ can be split into a pair of a saddle and a $k$-saddle {\rm(resp.} $\partial$-saddle and $k/2$-$\partial$-saddle{\rm)} with a multi-saddle separatrix between them by an arbitrarily small $C^r$-perturbation.
\end{lemma}


\subsection{Proof of Theorem~\ref{main:f}}
%
Fix a flow $v \in \mathcal{MS}_{k/2,l}^r(S)^c$.
Corollary~\ref{cor:Morse_dic_ms} implies that either $\mathrm{P}_{\mathrm{ms}}(v) \not\subset \partial S$ or there are multi-saddles that are not topologically hyperbolic.
The finite existence of limit cycles implies the isolated property of limit cycles.

We claim that the density holds.
Indeed, suppose that there are multi-saddles on the boundary which are not topologically hyperbolic.
From Lemma~\ref{lem:multi-saddle_ms}, we can split all topologically non-hyperbolic multi-saddles into saddles and $\partial$-saddles with multi-saddle separatrices by arbitrarily small $C^r$-perturbations as in Figure~\ref{split_multi_saddle}.
Therefore we may assume that $v$ has no topologically non-hyperbolic multi-saddles except one $1$-$\partial$-saddle.
As the proof of Lemma~\ref{m4}, by arbitrarily small $C^r$-perturbations, we can annihilate any multi-saddle separatrices outside of $\partial S$ except $\gamma$.
Then the resulting flow is $p$-unstable.
Suppose that any multi-saddles on the boundary are topologically hyperbolic.
From Lemma~\ref{lem:multi-saddle_ms}, we can split all topologically non-hyperbolic multi-saddles into saddles and $\partial$-saddles with multi-saddle separatrices by arbitrarily small $C^r$-perturbations as in Figure~\ref{split_multi_saddle}.
Therefore we may assume that $v$ has multi-saddle separatrices.
Fix a multi-saddle separatrix $\gamma \not\subset \partial S$.
As the proof of Lemma~\ref{m4}, by arbitrarily small $C^r$-perturbations, we can annihilate any multi-saddle separatrices outside of $\partial S$ except $\gamma$.
Then the resulting flow is $h$-unstable.

We claim that the openness holds.
Indeed, suppose that $v$ is $p$-unstable.
The condition of fixing the sum of indices of parabolic sectored singular points implies that each topological hyperbolic singular point is preserved by any small perturbations in $\mathcal{Q}_{k/2,l,**}^r(S)$.
Since the non-existence of multi-saddle separatrices outside of $\partial S$ is open, the $1$-$\partial$-saddle are preserved by any small perturbations in $\mathcal{MS}_{k/2,l}^r(S)^c$.
This means that $p$-unstable flows in $\mathcal{MS}_{k/2,l}^r(S)^c$ forms an open subset.
Suppose that $v$ is $h$-unstable.
As above, each topological hyperbolic singular point is preserved by any small perturbations in $\mathcal{Q}_{k/2,l,**}^r(S)$.
From Lemma~\ref{lem:Morse_dic}, the unique existence of a multi-saddle separatrix outside of $\partial S$ is open in $\mathcal{MS}_{k/2,l}^r(S)^c$.
This means that $h$-unstable flows in $\mathcal{MS}_{k/2,l}^r(S)^c$ forms an open subset.


%
\subsection{Proof of Theorem~\ref{th:generic_MS_trans}}

From Theorem~\ref{th:open_dense_MS}, the subset of Morse-Smale flows in $\mathcal{Q}_{k/2,l,**}(S)$ are open dense in $\mathcal{Q}_{k/2,l,**}(S)$.
Theorem~\ref{main:g} says that any generic Morse-Smale-like flows in $\mathcal{Q}_{k/2,l,**}(S)$ are $h$-unstable or $p$-unstable.
All homoclinic separatrices of saddles in $h$-unstable flows are listed on the left in Figure~\ref{trans_limit_all}, and all heteroclinic separatrices of $\partial$-saddles in limit circuits of $h$-unstable flows are listed on the right in Figure~\ref{trans_limit_all}.
Since each heteroclinic multi-saddle separatrix connects between saddles and $\partial$-saddles, all generic transitions between Morse-Smale flows in $\mathcal{Q}_{k/2,l,**}(S)$ are listed in Figure~\ref{trans_all}--\ref{trans_limit_all}.



\section{Observation for a complete invariant of Morse-Smale-like flows}

In the previous sections, we describe generic intermediate flows between gradient flows and between Morse-Smale flows using Morse-Smale-like flows.
Finally, we observe that such intermediate flows can be classified by a finite complete invariant of Morse-Smale-like flows.
Indeed, a Poincar\'e-Bendixson theorem for a flow with finitely many singular points and  \cite{yokoyama2017decompositions}*{Corollary 11.11} imply the following statement.

\begin{lemma}
Each connected component of $S - \mathop{\mathrm{BD}}_1(v) \subseteq \mathrm{P}(v)$ for a Morse-Smale-like flow $v$ on a compact surface $S$ is one of the following statements:
\\
$(1)$
An open trivial flow box in $\mathrm{P}(v)$, whose orbit space is an interval,
\\
$(2)$
An open annulus in $\mathrm{P}(v)$, whose orbit space is a circle,

Moreover, the $\omega$-limit {\rm(resp.} $\alpha$-limit{\rm)} set of a point in a connected component in $\mathrm{P}(v) \setminus \mathop{\mathrm{BD}}_1(v)$ is the $\omega$-limit {\rm(resp.} $\alpha$-limit{\rm)} set of any point in the connected component, which is either a singular point or a limit circuit.
\end{lemma}

Since the complement of $\mathop{\mathrm{BD}}_1$ consists of either a disk or an annulus, the previous lemma implies Theorem~\ref{main:g}.

\section{Acknowledgements}
The first author is partially supported by Russian Science Foundation 21-11-00355 project, the  second author is partially supported by JSPS Kakenhi Grant Number 20K03583

\bibliographystyle{amsplain}
\bibliography{TCMnew.bib}

\end{document}